\documentclass[12pt]{amsart}  	

\usepackage{geometry}  

\usepackage[all]{xy}						
\usepackage{bbm}
\usepackage{amssymb}
\usepackage{amscd,latexsym,amsthm,amsfonts,amssymb,amsmath,amsxtra}
\usepackage[mathscr]{eucal}
\usepackage[colorlinks]{hyperref}
\usepackage{mathtools}
\usepackage{cite}
\usepackage{chngcntr}
\numberwithin{equation}{subsection}
\newcounter{keepeqno}
\newenvironment{num}{\setcounter{keepeqno}{\value{equation}}
\begin{list}
{(\theequation)}{\usecounter{equation}}%
\setcounter{equation}{\value{keepeqno}}}{\end{list}}

\hypersetup{linkcolor=black, citecolor=black}

\newcommand{\SF}{{\mathcal{SF}}}

\newcommand{\BC}{{\mathbb {C}}}

\newcommand{\BG}{{\mathbb {G}}}

\newcommand{\BN}{{\mathbb {N}}}

\newcommand{\BR}{{\mathbb {R}}}

\newcommand{\BZ}{{\mathbb {Z}}}

\newcommand{\CE}{{\mathcal {E}}}

\newcommand{\CL}{{\mathcal {L}}}

\newcommand{\CN}{{\mathcal {N}}}

\newcommand{\CS}{{\mathcal {S}}}

\newcommand{\CU}{{\mathcal {U}}}

\newcommand{\CW}{{\mathcal {W}}}
\newcommand{\CX}{{\mathcal {X}}}

\newcommand{\CZ}{{\mathcal {Z}}}

\newcommand{\Fg}{{\mathfrak {g}}}

\newcommand{\Fw}{{\mathfrak {w}}}

\newcommand{\RI}{{\mathrm {I}}}

\newcommand{\RM}{{\mathrm {M}}}

\newcommand{\RO}{{\mathrm {O}}}

\newcommand{\Ad}{{\mathrm{Ad}}}

\newcommand{\GL}{{\mathrm{GL}}}

\newcommand{\Hom}{{\mathrm{Hom}}}

\renewcommand{\Im}{{\mathrm{Im}}}

\newcommand{\Ind}{{\mathrm{Ind}}}

\renewcommand{\Re}{{\mathrm{Re}}}

\newcommand{\SL}{{\mathrm{SL}}}

\newcommand{\SO}{{\mathrm{SO}}}

\newcommand{\Sym}{{\mathrm{Sym}}}
\newcommand{\sgn}{{\mathrm{sgn}}}

\newcommand{\udl}{\underline}
\newcommand{\wh}{\widehat}

\newcommand{\bs}{\backslash}

\def\std{\rm std}

\def\Irr{\mathrm{Irr}}

\def\ss{\mathrm{ss}}

\def\temp{\mathrm{temp}}

\def\Vogan{\mathrm{Vogan}}

\def\rel{\mathrm{rel}}

\newtheorem{thm}{Theorem}[subsection]
\newtheorem{defin}[thm]{Definition}

\newtheorem{pro}[thm]{Proposition}
\newtheorem{lem}[thm]{Lemma}
\newtheorem{cor}[thm]{Corollary}

\newtheorem{conjec}{Conjecture}

\makeatletter

\newcommand{\Rmnum}[1]{\expandafter\@slowromancap\romannumeral #1@}
\makeatother

\begin{document}

\title[local GGP over Archimedean fields]{The local Gross-Prasad conjecture over Archimedean local fields}

\author{Cheng Chen}
\address{Institut de Math\'ematiques de Jussieu–Paris Rive Gauche, CNRS, Paris, France}
\email{cheng.chen@imj-prg.fr}
\keywords{local Gross-Prasad conjecture, branching problem, multiplicity formula}
\subjclass[2020]{Primary 22E50 22E45; Secondary 20G20}
\maketitle
\begin{abstract}
Following the approach of C. M\oe glin and J.-L. Waldspurger in \cite{moeglin2012conjecture}, this article proves the local Gross-Prasad conjecture over $\BR$ and $\BC$ based on the tempered cases in \cite{luothesis}\cite{chen2022local}.
\end{abstract}
\tableofcontents
\section{Introduction}
In \cite{gross1992decomposition}\cite{gross1994irreducible}, 
B. Gross and D. Prasad formulated a conjecture   on the local multiplicity for Bessel models of special orthogonal groups over a local field of characteristic $0$, know as the \textit{local Gross-Prasad conjecture}. When the local field is non-Archimedean, the conjecture was proved in \cite{moeglin2012conjecture} based on the tempered cases proved in \cite{waldspurger2010formule}\cite{waldspurger2012calcul}\cite{waldspurger2012conjecture}\cite{waldspurger2012formule}\cite{waldspurger2012variante}. This paper proves the local Gross-Prasad conjecture over Archimedean local fields. In particular, the proof over the real field follows M\oe glin and Waldspurger's approach and is based on the tempered cases proved in \cite{luothesis}\cite{chen2022local}.

There are some recent applications of the local Gross-Prasad conjecture. The paper \cite{jiang2020arthur} takes it as an input to prove one direction of the global Gross-Prasad conjecture, and the paper \cite{jiang2022arithmetic} uses the local Gross-Prasad conjecture to develop the theory of arithmetic wavefront sets for irreducible admissible representations of classical groups. We refer to the ICM report of R. Beuzart-Plessis \cite{beuzartplessis} for a general discussion of the significance of the local Gross-Prasad conjecture in arithmetics. 

The local Gross-Prasad conjecture is set up as following: Let $F$ be a local field of characteristic $0$, and $(W, V)$ be a pair of non-degenerate quadratic spaces over $F$ such that the orthogonal complement $W^{\perp}$ of $W$ in $V$ is odd-dimensional and split over $F$. We let $\mathsf G$ be the algebraic group $\mathsf{SO}(W)\times \mathsf{SO}(V)$ over $F$ and take its subgroup $\mathsf H=
\Delta\mathsf{SO}(W)\ltimes \mathsf N$, where $\Delta \mathsf{SO}(W)$ is the image of the diagonal embedding $\mathsf{SO}(W)\hookrightarrow \mathsf{SO}(W)\times\mathsf{SO}(V)$ and $\mathsf N$ is the unipotent part of a parabolic subgroup stabilizing a full totally isotropic flag on $W^{\perp}$. We fix a generic character $\xi_N$ of $N=\mathsf N(F)$ that uniquely extends to a character $\xi$ of $H=\mathsf H(F)$.
For every irreducible admissible representation $\pi$ of $G=\mathsf G(F)$ (we require the representation to be Casselman-Wallach when $F$ is Archimedean), we define the multiplicity
\[
m(\pi):=\dim\Hom_{H}(\pi|_{H},\xi).
\]
It was proved in  \cite{aizenbud2010multiplicity}\cite{gan2012symplectic}\cite{waldspurger2012variante} over non-Archimedean fields and in \cite{sun2012multiplicity}\cite{jiang2010uniqueness} over Archimedean fields that
\[
m(\pi)\leqslant 1.
\]
This result is known as the \textit{multiplicity-one theorem}. The local Gross-Prasad conjecture is a refinement of the multiplicity-one theorem that takes representations of pure inner forms of $\mathsf G$ into consideration.

For every $\alpha\in H^1(F, \mathsf H)\hookrightarrow H^1(F,\mathsf G)$, the inner twists of $\mathsf G,\mathsf H$ by $\alpha$ give pure inner forms $\mathsf{G}_{\alpha},\mathsf H_{\alpha}$, respectively. Then $\mathsf G_{\alpha}=\mathsf{SO}(W_{\alpha})\times \mathsf{SO}(V_{\alpha})$ and $\mathsf H_{\alpha}=\Delta\mathsf{SO}(W_{\alpha})\ltimes \mathsf N$, where $W_{\alpha}$ is the inner twist of $W$ by $\alpha\in H^1(F, \mathsf H)=H^1(F,\mathsf{SO}(W_{\alpha}))$ and $V_{\alpha}=W_{\alpha}\perp S$. Let $\xi_{\alpha}$ be the character of $H_{\alpha}=\mathsf H_{\alpha}(F)$ obtained by the extension of $\xi_N$. 
For every irreducible admissible representation $\pi$  of $G_{\alpha}=\mathsf G_{\alpha}(F)$ (we require the representation to be Casselman-Wallach when $F$ is Archimedean), we extend the definition of multiplicity by setting
\[
m(\pi):=\dim\Hom_{H_{\alpha}}(\pi|_{H_{\alpha}},\xi_{\alpha}).
\]

 For every local $L$-parameter $\phi:\CW_F\to {}^L\mathsf G$, we denote by $\Pi_{F,\phi}(\mathsf G)$ the corresponding $L$-packet, which consists of finitely many irreducible admissible representations of $\mathsf{G}(F)$, which are Casselman-Wallach when $F$ is Archimedean. For every $\alpha\in H^1(F,\mathsf G)$. the Langlands dual group ${}^L \mathsf G_{\alpha}$ of $\mathsf G_{\alpha}$ is isomorphic to that of $\mathsf G$, so $\phi$ also represents a local $L$-parameter of $\mathsf G_{\alpha}$. Following the work of D. Vogan (\cite{vogan1993local}), we can define the Vogan $L$-packet associated to $\phi$ as
\[
\Pi_{F,\phi}^{\Vogan}:=\bigsqcup_{\alpha\in H^1(F,\mathsf G)}\Pi_{\phi}(\mathsf G_{\alpha}).
\]
The $L$-parameter $\phi$ is called \textit{tempered} if $\Im(\phi)$ is bounded. The $L$-parameter $\phi$ is called \textit{generic} if there is a generic representation in $\Pi_{F,\phi}^{\Vogan}$. In particular, tempered parameters are generic.

When $\phi$ is generic, it was conjectured by Vogan and known over Archimedean local fields (\cite[Theorem  6.3]{vogan1993local}), that, fixing a Whittaker datum of $\{\mathsf G_{\alpha}\}_{\alpha\in H^1(F, \mathsf G)}$, there is a bijection 
\[
\pi\in\Pi^{\Vogan}_{F,\phi}\longleftrightarrow \eta_{\pi}\in\widehat{\CS_{\phi}}.
\]
Here $\widehat{\CS_{\phi}}$ is the set of (complex) characters of  component group 
\[
\CS_{\phi}:=\pi_0(\mathrm{Cent}_{\wh{G}}(\Im(\phi))),
\] where $\mathrm{Cent}_{\wh{G}}(\Im(\phi))$ is the centralizer of the image $\Im(\phi)$ in the dual group $\wh{\mathsf G}$. Gross and Prasad suggested that one may consider the \textit{relevant Vogan packet}, defined as
\[
\Pi^{\Vogan}_{F,\phi,\rel}:=\bigsqcup_{\alpha\in H^1(F,\mathsf H)}\Pi_{F,\phi}(\mathsf G_{\alpha})\subset \Pi^{\Vogan}_{F,\phi}.
\]
In particular, the multiplicity $m(\pi)$ is well-defined for representations in $\Pi_{F,\phi,\rel}^{\Vogan}$.

\begin{conjec}[\cite{gross1992decomposition}\cite{gross1994irreducible}]\label{conj: GP in introduction}
With the notions above, the following two statements hold.
\begin{enumerate}
    \item (Multiplicity-one) For every generic parameter $\phi$ of $\mathsf{G}$, we have
    \[
    \sum_{\pi\in \Pi^{\Vogan}_{F,\phi,\rel}}m(\pi)=1.
    \] 
    This implies that there is an unique representation $\pi\in \Pi_{F,\phi,\rel}^{\Vogan}$ such that $m(\pi)=1$.
    \item (Epsilon-dichotomy) Fix the Whittaker datum of $\{\mathsf G_{\alpha}\}_{\alpha\in H^1(F,\mathsf G)}$ as \cite[(6.3)]{gross1994irreducible}. The unique representation $\pi\in \Pi^{\Vogan}_{F,\phi,\rel}$ such that $m(\pi)=1$ can be characterized as
    \[
    \eta_{\pi}=\eta_{\phi},
    \]
    where $\eta_{\phi}$ is defined in (\ref{equ: definition of character}).
\end{enumerate}
\end{conjec}

When $F$ is non-Archimedean and $\phi$ is tempered,  Waldspurger proved the conjecture in \cite{waldspurger2010formule}\cite{waldspurger2012calcul}\cite{waldspurger2012conjecture}\cite{waldspurger2012formule}\cite{waldspurger2012variante}. M\oe glin and Waldspurger completed the proof of  Conjecture \ref{conj: GP in introduction} for generic parameters based on the results in the tempered cases.

When $F=\BR$ and the parameter $\phi$ is tempered, Z. Luo proved the multiplicity-one part of  Conjecture \ref{conj: GP in introduction} in \cite{luothesis} following the work of R. Beuzart-Plessis in \cite{beuzart2019local}. The author and Luo proved the epsilon-dichotomy part of Conjecture \ref{conj: GP in introduction} in \cite{chen2022local} by a simplification of Waldspurger's approach.

The main result of the paper is the following.
\begin{thm}
When $F=\BR$ or $\BC$, Conjecture \ref{conj: GP in introduction} holds for generic parameters. 
\end{thm}

The proof over $\BC$ is done by construction based on results in \cite{gourevitch2019analytic} and the proof over $\BR$ follows the strategy in \cite{moeglin2012conjecture}.  The proof consists of a structure theorem (Proposition \ref{pro: real ind}) for representations in generic packets and a multiplicity formula (Theorem \ref{thm: reductive}). With these results, we can reduce all situations of the conjecture into the tempered cases.

In Section \ref{section: classification}, we prove the structure theorem using the standard module conjecture. The proof of the multiplicity formula, however, is more intricate. Following \cite{moeglin2012conjecture}, this requires a formula for reduction to basic cases and two multiplicity formulas that establish inequalities needed to prove the basic cases.

In the basic case, one inequality of the multiplicity formula is proved using orbit analysis (Section \ref{section: multiplicity formula}). The proof of the other inequality is expected to be completed using harmonic analysis in Section \ref{section: harmonic analysis}. The formula for reduction to the basic cases, which is an equality, can be established by proving two inequalities in a manner similar to the inequalities in the basic case. The non-Archimedean counterpart is discussed in \cite[Section 2]{moeglin2012conjecture}, \cite[Sections 1.4–1.6]{moeglin2012conjecture}, and \cite[Sections 1.7–1.8]{moeglin2012conjecture}.

It is worth mentioning that there is a parallel conjecture for unitary groups formulated by W. Gan, Gross, and Prasad. Over non-Archimedean local fields, the conjecture for tempered parameters was treated by Beuzart-Plessis in  \cite{beuzart2014expression}\cite{beuzart2016conjecture}; Based on the tempered cases, Gan and A. Ichino proved the conjecture for generic parameters in \cite{gan2016gross}. Over Archimedean local fields,  Beuzart-Plessis proved the multiplicity-one part of the conjecture in \cite{beuzart2019local} for tempered parameters using local trace formula and endoscopy. Xue completed the proof for tempered cases in \cite{xue1bessel} using theta correspondence and proved the generic cases in \cite{xue2020bessel}.

Although it is not necessary for the proof for the local Gan-Gross-Prasad conjecture, the multiplicity formula (Theorem \ref{thm: reductive})  also works for reducible representations obtained from parabolic induction. This result can be applied to the study of local descents in my joint work with D. Jiang, D. Liu, L. Zhang \cite{chen2024arithmetic}.

\textit{Organization.} 
In Section \ref{section: GP conjecture}, we recall the statement of the local Gross-Prasad conjecture following \cite{gross1992decomposition}\cite{gross1994irreducible}. In Section \ref{section: complex}, we work over the complex field $\BC$. We follow the observation in \cite[\S 11]{gross1992decomposition} and prove the conjecture by constructing an explicit functional of the representation $\pi_V\boxtimes \pi_W$ using the results in \cite{gourevitch2019analytic}. 

In Sections \ref{section: classification}–\ref{section: real}, we work over the real field 
$\BR$. Section \ref{section: classification} provides a structure theorem for representations in generic packets, using a sufficient condition for irreducibility. In Section \ref{section: real}, we reduce the conjecture to the tempered cases by employing a multiplicity formula, following the approach in \cite{moeglin2012conjecture}.

For the basic case of the multiplicity formula, we prove one inequality using representation theory and orbit analysis (Section \ref{section: multiplicity formula}) and the other using harmonic analysis (Section \ref{section: harmonic analysis}). Additionally, in Sections \ref{section: multiplicity formula}–\ref{section: harmonic analysis}, we establish a formula that reduces the multiplicity to the basic cases.

\textbf{Acknowledgement.}  I would like to express my deepest gratitude to my advisor, Prof. Dihua Jiang, for encouraging me to explore this subject and for providing invaluable guidance throughout my research and the writing of this article. This work was partially supported by the National Science Foundation grants DMS-1901802 and DMS-2200890. Additionally, this project received funding from the European Union’s Horizon 2020 research and innovation programme under the Marie Skłodowska-Curie grant agreement No. 101034255.

I am grateful to Zhilin Luo for his patience and collaboration on the proof of the tempered cases in \cite{chen2022local}. I would also like to thank Fangyang Tian for insightful discussions on technical issues related to Casselman-Wallach representations. My thanks extend to Chen Wan for offering significant feedback that prompted the reorganization of Section \ref{section: complex}. I thank Rui Chen and Jialiang Zou for their collaboration on the Fourier-Jacobi case, which was instrumental in developing Theorem \ref{thm: CCZ}. I also wish to thank Binyong Sun for facilitating my visit to Zhejiang University in 2024 and for engaging in helpful discussions during my time there.

Finally, I am sincerely grateful to the anonymous referees for their constructive suggestions, which greatly improved this paper.

\section{Local Gross-Prasad Conjecture}\label{section: GP conjecture}
In this section, we review the local Gross-Prasad conjecture over Archimedean local fields following \cite{gross1992decomposition} and \cite{gross1994irreducible}.
\subsection{Gross-Prasad triples}\label{section: Gross-Prasad triple}

Let $F=\BR$ or $\BC$ and $(W, V)$ be a pair of non-degenerate quadratic spaces over $F$. The pair $(W, V)$ is called \textit{relevant} if and only if there exists an anisotropic line $D$ and a non-degenerate even-dimensional split quadratic space $Z$  over $F$ such that
\[
V= W{\perp} D\perp Z.
\]

We denote $r=\frac{\dim Z}{2}$. There exists a basis $\{z_{i}\}_{i=\pm 1}^{\pm r}$ of $Z$ such that
\[
q(z_i,z_j)=\delta_{i,-j}, \quad\forall  i,j\in \{\pm 1,\dots,\pm r\},
\]
where $q$ is the quadratic form on $V$. We denote by $\mathsf P_V$  the parabolic subgroup of the special orthogonal group $\mathsf{SO}(V)$ stabilizing the  totally isotropic flag 
\begin{equation}\label{equ: totally isotropic flag}
\langle z_r\rangle\subset \langle z_r,z_{r-1}\rangle\subset \cdots \subset \langle z_r,\dots,z_1\rangle.
\end{equation}
We take $\mathsf P_V=\mathsf M_V\cdot \mathsf N$ to be its Levi decomposition. In particular, the Levi subgroup  $\mathsf M_V\simeq \mathsf{SO}(W\oplus D)\times \GL_1^r$.

Let $\mathsf G=\mathsf{SO}(W)\times \mathsf{SO}(V)$. We identify $\mathsf N$ as a subgroup of $\mathsf G$ via the embedding $\mathsf{SO}(V)\hookrightarrow 1\times \mathsf{SO}(V)$. We set $\Delta\mathsf{SO}(W)$ as the image of the diagonal embedding $\mathsf{SO}(W)\hookrightarrow \mathsf G$. Then $\Delta\mathsf{SO}(W)$ acts on $\mathsf N$ by adjoint action of $\mathsf{SO}(W)\subset \mathsf M_V$ and we set
\[
\mathsf H=\Delta\mathsf{SO}(W)\ltimes \mathsf{N}.
\]

We  define a morphism $\lambda:\mathsf{N}\to \BG_a$ by
\[
\lambda(n)=\sum_{i=0}^{r-1}q(z_{-i-1},nz_i),\quad n\in \mathsf{N}.
\]
Then $\lambda$ is $\Delta\SO(W)$-conjugation invariant and hence $\lambda$ admits an unique extension to $\mathsf H$ that is trivial on $\Delta\mathsf{SO}(W)$. We still denote this character by $\lambda$. Let $\lambda_{F}: \mathsf H(F)\to F$ be the induced morphism on $F$-rational points. We define an unitary character of $H=\mathsf H(F)$ by
\[
\xi(h)=\lambda_F(h),\quad h\in H,
\]
where $\psi$ is a fixed additive (unitary) character $\psi$ of $F$.

Then the triple $(\mathsf G, \mathsf H,\xi)$ is called the \textit{Gross-Prasad triple} associated with the relevant pair $(W,V)$.
\subsection{Vogan $L$-packets}
In this subsection, we are going to recall the notion of Vogan $L$-packets for special orthogonal groups over Archimedean local fields following \cite{vogan1993local} and review the definition of the relevant Vogan $L$-packet following \cite{gross1992decomposition}\cite{gan2012symplectic}.

For any reductive algebraic group $\mathsf G$ over a local field $F$, we denote by $\wh{\mathsf G}$ the dual group of $\mathsf G$ and by ${}^L\mathsf G$ the Langlands dual group of $\mathsf G$. It was established by Langlands in \cite{langlands1973the} that every local $L$-parameter $\phi: \CL_{F}\to {}^L\mathsf G$ gives a local $L$-packet $\Pi_{F,\phi}(\mathsf G)$, which consists of a finite set of irreducible admissible representations of $G=\mathsf G(F)$. In particular, when $F$ is Archimedean, the representations in the packet are Casselman-Wallach ( \cite{casselman1989canonical}\cite{wallach1994real}), which means that they are smooth Fr\'echet representations of moderate growth and the associated Harish-Chandra modules are admissible.
 

A \textit{pure inner forms}  $\mathsf G_{\alpha}$ is an inner twist of $\mathsf G$ by $\alpha\in H^1(F,\mathsf G)$. Since pure inner forms of $\mathsf G$ share the same dual group, every local $L$-parameter $\phi: \CL_{F}\to {}^L\mathsf G$ of $\mathsf G$ can be viewed as an $L$-parameter for any pure inner form $\mathsf G_{\alpha}$.
Hence, one can define the \textit{Vogan $L$-packet}  as
\[
\Pi^{\Vogan}_{F,\phi}:=\bigsqcup_{\alpha\in H^1(F,\mathsf G)}\Pi_{F,\phi}(\mathsf G').
\]

Now we consider reductive group $\mathsf G$ with a quasi-split pure inner form. A \textit{Whittaker datum} $\Fw$ for $\mathsf G$ is a triple $(\mathsf G',\mathsf B',\psi')$ where $\mathsf  G'$ is a quasi-split pure inner form of $\mathsf G$, $\mathsf  B'$ is a Borel subgroup of $\mathsf  G'$, and $\psi'$ is a generic character of the unipotent radical $N'=\mathsf N'(F)$ of $\mathsf B'(F)$.  A representation $\pi'$ of $\mathsf G'(F)$ is called $\Fw$-generic if $\Hom_{N'}(\pi'|_{N'},\xi')\neq 0$. An $L$-parameter $\phi$ is called ($\Fw$-)\textit{generic} if the Vogan $L$-packet contains a generic representation. As argued in \cite[\S 18]{gan2012symplectic}, the genericity of an $L$-parameter is independent of the choice of the Whittaker datum.

From \cite{vogan1993local}, when $F$ is Archimedean, fixing a generic $L$-parameter $\phi$ and a Whittaker datum $\Fw$ of $\mathsf G$, there is a bijection 
\begin{equation}\label{equ: parameterization by character}
\pi\in\Pi^{\Vogan}_{F,\phi}\mapsto \eta_\pi\in \Pi(\CS_{\phi}),
\end{equation}
where $\Pi(\CS_{\phi})$ is the set of characters of the \textit{component group}
\[
\CS_{\phi}:=\pi_0(\mathrm{Cent}_{\wh{\mathsf G}}(\Im (\phi))).
\]
Therefore, we can parameterize representations in Vogan packets with characters $\eta:\CS_{\phi}\to \{\pm 1\}$.

Now we return to the setting in Section \ref{section: Gross-Prasad triple}. For $\alpha\in H^1(F,\mathsf H)=H^1(F,\mathsf{SO}(W))$, we denote by $W_{\alpha}$ the inner twist of $W$ by $\alpha$ and set $V_{\alpha}=W_{\alpha}\perp D\perp Z$. Then the inner twists of $\mathsf G$ and  $\mathsf H$  by $\alpha\in H^1(F,\mathsf H)\subset H^1(F,\mathsf G)$  are
\[
\mathsf G_{\alpha}=\mathsf{SO}(V_{\alpha})\times \mathsf{SO}(W_{\alpha})\text{ and } \mathsf H_{\alpha}=\Delta \mathsf{SO}(W_{\alpha})\ltimes \mathsf N.
\]
Together with the character $\xi_{\alpha}:\mathsf N(F)\to \BC$ obtained by the extension of $\xi_N$, we obtain the Gross-Prasad triple associated to the relevant pair $(W_{\alpha},V_{\alpha})$. The \textit{relevant Vogan packet} is defined by
\begin{equation}
\Pi_{F,\phi,\rel}^{\Vogan}:=\bigsqcup_{\alpha\in H^1(F,\mathsf H)}\Pi_\phi{(\mathsf G_{\alpha})}.
\end{equation}

It is a subset of $\Pi_{F,\phi}^{\Vogan}$ and thus can be parameterized with a subset of $\Pi(\CS_{\phi})$ via \eqref{equ: parameterization by character}.

\subsection{The conjecture}
In this subsection, we review the statement of the local Gross-Prasad conjecture formulated in \cite{gross1992decomposition}\cite{gross1994irreducible}.

Let $(W,V)$ be a relevant pair over an Archimedean local field $F$ and $(\mathsf G,\mathsf H,\xi)$ be the Gross-Prasad triple associate to it.
 For an irreducible Casselman-Wallach representation $\pi$ of $G=\mathsf G(F)$, we set $H=\mathsf H(F)$ and define the multiplicity
\begin{equation}
    m(\pi):=\dim \Hom_{H}(\pi,\xi).
\end{equation}

From the multiplicity-one theorem established in \cite{sun2012multiplicity}\cite{jiang2010uniqueness}, we have
\[
m(\pi)\leqslant 1.
\]

The local Gross-Prasad conjecture (Conjecture \ref{conj: GP in introduction}) studies the refinement behavior of the multiplicity $m(\pi)$ in a relevant Vogan $L$-packet, which shows that there is exactly one representation $\pi_{\phi}$ in $\Pi_{F,\rel,\phi}^{\Vogan}$ with multiplicity equal to $1$ and the character $\eta_{\pi_{\phi}}:\CS_{\phi}\to \{\pm 1\}$ attached to  $\pi_{\phi}$ is equal to an explicit character $\eta_{\phi}$.

 For generic character $\phi=\phi_V\times \phi_W$ of $\mathsf G$, the character 
 \[
 \eta_{\phi}=\eta_{\phi_V}^W\times \eta_{\phi_W}^V:\CS_{\phi_V}\times \CS_{\phi_W}\to \{\pm 1\}
 \]was constructed explicitly in \cite[\S 10]{gross1992decomposition}.  For every element $s\in \CS_{\phi_W}\times \CS_{\phi_V}$, set 
\begin{equation}\label{equ: definition of character}
\begin{aligned}
&\eta_{\phi_V}^W\left(s_V\right)=\operatorname{det}\left(-\mathrm{Id}_{\mathrm{M}_V^{s_V=-1}}\right)^{\frac{\operatorname{dim} \mathrm{M}_W}{2}} \cdot \operatorname{det}\left(-\mathrm{Id}_{\mathrm{M}_W}\right)^{\frac{\operatorname{dim} \mathrm{M}_V^{s_V=-1}}{2}} \cdot \varepsilon\left(\frac{1}{2}, \mathrm{M}_V^{s_V=-1} \otimes \mathrm{M}_W, \psi\right)\\
&\eta_{\phi _W}^V\left(s_W\right)=\operatorname{det}\left(-\operatorname{Id}_{\mathrm{M}_W^{s_W=-1}}\right)^{\frac{\operatorname{dim} \mathrm{M}_V}{2}} \cdot \operatorname{det}\left(-\operatorname{Id}_{\mathrm{M}_V}\right)^{\frac{\operatorname{dim}{\mathrm{M}_W^{s_W=-1}}}{2}} \cdot \varepsilon\left(\frac{1}{2}, \mathrm{M}_W^{s_W=-1} \otimes \mathrm{M}_V, \psi\right)
\end{aligned}
\end{equation}
Here $\RM_V$ and $\RM_W$ are the space of the standard representation of ${}^L\SO(V)$ and ${}^L\mathsf{SO}(W)$, respectively. The space $\RM_V^{s_V=-1}$ denotes the $s_V=(-1)$-eigenspace of $\RM_V$ and $\varepsilon(\dots)$ is the local root number defined by Rankin-Selberg integral (\cite{jacquet2009Archimedean}).

The result of this article is the following.

When $F=\BC$, the relevant Vogan $L$-packet $\Pi^{\Vogan}_{F,\phi,\rel}$ contains only one element. Hence, Part (1) of the conjecture implies Part (2) of the conjecture. We will prove the following theorem by constructing a non-zero element in $\Hom_{H}(\pi,\xi)$  in Section \ref{section: complex}.

\begin{thm}\label{thm: complex}
When $F=\BC$, Conjecture \ref{conj: GP in introduction} holds.
\end{thm}

When $F=\BR$, in \cite{luothesis},  following the work of Waldspurger (\cite{waldspurger2010formule}\cite{waldspurger2012conjecture}) and Beuzart-Plessis (\cite{beuzart2019local}), Luo proved  Part (1) of Conjecture \ref{conj: GP in introduction} when the parameter $\phi$ is tempered.  In \cite{chen2022local}, by simplifying Waldspurger's approach (\cite{waldspurger2010formule}\cite{waldspurger2012calcul}\cite{waldspurger2012conjecture}\cite{waldspurger2012formule}\cite{waldspurger2012variante}), the author and Luo proved   Part (2) of Conjecture \ref{conj: GP in introduction} when the parameter $\phi$ is tempered. The main result in Section \ref{section: real} is to prove Theorem \ref{thm: reductive} that implies the following theorem based on the Conjecture \ref{conj: GP in introduction} for tempered parameters.
\begin{thm}\label{thm: real}
When $F=\BR$, Conjecture \ref{conj: GP in introduction} holds.
\end{thm}

\section{Integral method and the proof for the complex case}\label{section: complex}

One of the main tools for proving Conjecture \ref{conj: GP in introduction} is the integral method. In particular, this is the only tool we would apply to prove Conjecture \ref{conj: GP in introduction} when $F=\BC$. When $F=\BC$ and $\dim V=\dim W+1$, Conjecture \ref{conj: GP in introduction} was proved by J. M\"ollers in \cite{mollers2017symmetry} using an equivalent method. In section \ref{section: complex}, we use some computation in \cite{mollers2017symmetry} and present the proof using the integral method following \cite{gourevitch2019analytic}. 

Let $F=\BR,\BC$. Let $\mathsf G$ be a quasi-split group over $F$ and $\mathsf H$ be a closed subgroup of $\mathsf G$ such that $\mathsf G/\mathsf H$ is absolutely spherical. Suppose there is a Borel subgroup $\mathsf B$ of $\mathsf G$ such that 
\[
\mathsf B\cap \mathsf H=1.
\]
Let $\mathsf T$ be the Levi component of $\mathsf B$. We set
\[
G=\mathsf G(F),\quad H=\mathsf H(F),\quad B=\mathsf B(F),\quad T=\mathsf T(F).
\]

Fix a unitary character $\psi$ of $F$. For an algebraic character $\lambda:\mathsf H\to \mathbb G_a$, we set 
 $\xi=\psi \circ \lambda_F$, which is a unitary character of $H$.

 As a consequence of the integral method in \cite{gourevitch2019analytic}, we have the following theorem.
\begin{thm}\label{thm: construction by meromorphic continuation}
Let $G, H, B, T$ as above. 
For every character $\sigma$ of $T$, we have
\[
\dim \Hom_H(\Ind_{B}^{G}(\sigma),\xi)\geqslant 1.
\]
\end{thm}

First, we construct a measure $\mu$ on $B\cdot H\subset G$ by setting $\mu=f(bh)dbdh$
where
\[
f(bh):=\delta_{B}^{-1/2}(b)\sigma^{-1}(b)\xi(h),\quad 
b\in B,\ h\in H.
\]

We can express the function $f$ in the form of
 \[
f(bh)=t^{\mu_1}\overline{t}^{\mu_2}e^{is_1\Re(\lambda(h))+s_2\Im(\lambda(h))}\quad\forall b=t\cdot n\in B=T\cdot N,\ h\in H
 \]
 for certain $s_1,s_2\in \BR$ and $\mu_1,\mu_2\in\Hom(\mathsf T,\BG_m)$.
Hence, for every differential operator $D$ on $B\times H$, the growth of $|Df|$ can be controlled by a polynomial. Therefore, $\mu$ is a tempered measure on $B\cdot H$, which is left-$(B,\delta_{B}^{1/2}\sigma)$-equivariant and right-$(\mathsf H(F),\xi)$-equivariant.
Because $\mathsf B$ is solvable, from \cite[Theorem B]{gourevitch2019analytic}, one can construct a left-$(B,\delta_{\mathsf B}^{1/2}\sigma)$-equivariant and right-$(H,\xi)$-equivariant distribution on $\mathsf G$.

From \cite{du1991representations} and the compactness of $B\bs G$, there is a one-to-one correspondence between $\Hom(\Ind_{B}^{G}(\sigma),\xi)$ and the space of left-$(B,\delta_{B}^{1/2}\sigma)$-equivariant and right-$(H,\xi)$-equivariant distributions on $G$.

Now we return to the Gross-Prasad conjecture over $F=\BC$. As argued in \cite[\S 11]{gross1992decomposition}, since there is exactly one representation in the relevant Vogan $L$-packet and this representation is a principal series, it suffices to verify that  $m(\pi)\geqslant 1$ for every principal series representation $\pi=\Ind_{B}^{G}(\sigma)$.  For this purpose, we verify $\mathsf B\cap \mathsf H=1$ when $(\mathsf G, \mathsf H,\xi)$ is the Gross-Prasad triple associated to a relevant pair $(W, V)$.

Set $\mathsf P_V=\mathsf M_V\cdot \mathsf N$ be the parabolic subgroup stabilizing the totally isotropic flag (\ref{equ: totally isotropic flag}) and the Levi subgroup $\mathsf M_V$ can be decomposed as $\mathsf M_V=\prod_{i=1}^r\GL(\BC\cdot z_i)\times \SO(V\oplus D)$. Let $\overline{\mathsf P}_V=\mathsf M_V\cdot \overline{\mathsf N}$ be the opposite parabolic subgroup of $P_V$.

Let $(\mathsf G',\mathsf H',\xi')$ be the Gross-Prasad triple associated to the relevant pair $(W,W\oplus D)$. From \cite[\S 6.2.4]{mollers2017symmetry}, there exists a Borel subgroup $\mathsf B'$ of $\mathsf G'=\SO(W\oplus D)\times \SO(W)$ such that $\mathsf B'\cap \mathsf H'=1$. We set $\mathsf B=\mathsf B'\cdot \prod_{i=1}^r\GL(\BC\cdot z_i)\cdot \mathsf B' \cdot (\bar{\mathsf N}\times 1)$. Consider the parabolic subgroup $\mathsf P=\mathsf P_V\times \SO(W)=\mathsf M\cdot (\mathsf N\times 1)$ of $\mathsf G$. Since 
$\prod_{i=1}^r\GL(\BC\cdot z_i)\mathsf B'$ and $\mathsf H'$ are subgroups of $\mathsf M=\mathsf M_V\times \SO(W)$ such that 
\[
\prod_{i=1}^r\GL(\BC\cdot z_i)\mathsf B'\cap \mathsf H'=1,
\]
we have
\[
\mathsf B\cap \mathsf H=\overline{\mathsf N}\cdot \prod_{i=1}^r\GL(\BC\cdot z_i)\mathsf B'\cap \mathsf H'\cdot \mathsf N=\GL(\BC\cdot z_i)\mathsf B'\cap \mathsf H'=1. 
\]
This completes the proof for Theorem \ref{thm: complex}.

\section{Representations in generic packets}\label{section: classification}
In this section,  we prove that, for every parameter $\phi$ of a special orthogonal group over $\BR$, there is a tempered $L$-parameter $\phi_0$ of a smaller special orthogonal group with decomposition $\phi=\phi^{\GL}\oplus\phi_0\oplus (\phi^{\GL})^{\vee}$, such that the parabolic induction 
\[
\pi_0\mapsto \sigma\ltimes \pi_0\] induces isomorphism before $\Pi_{\phi_0}^{\Vogan}$ and $\Pi_{\phi}^{\Vogan}$, where $\sigma$ is the unique representation in the packet $\Pi_{\phi^{\GL}}$. 

Let $V$ be a non-degenerate quadratic space over $\BR$. It is well-known that an $L$-parameter $\phi_V$ of $\SO(V)$ is generic if and only if 
 the adjoint $L$-function $L(s,\phi_V,\Ad)$ is holomorphic at $s=1$ (\cite[Conjecture 2.6]{gross1992decomposition} and the remark after). Based on this property, we first compute an equivalent condition for $\phi_V$ to be generic.
 
\begin{defin}
Given a generic $L$-parameter $\phi_{V}:\CW_{\BR}\to {}^L\SO(V)$. We denote by $\phi_V^{\ss}$ the \textit{semisimplification} of $\phi_V$, that is, the semisimple representation on $M_V$ defined by the composition $\phi_V$ with the standard representation $\std_V:{}^L\SO(V)\to \GL(M_V)$.
\end{defin}

Given an $L$-parameter $\phi_V$, its  semi-simplification $\phi_V^{\ss}$ can be decomposed  as following
\begin{equation}\label{equ: semisimplification}
\phi_V^{\ss}=\bigoplus
|\cdot|^{s_{V,i}^1}\phi_{l_{V,i}}^1+\bigoplus
|\cdot|^{s_{V,i}^2}\phi_{{m_{V,i}}}^2.
\end{equation}
Here $\phi_{l_{V,i}}^1\ (l_{V,i}\in\BZ)$ is a one-dimensional representation of $\CW_{\BR}=\BC\cup \BC \cdot j\ (j^{2}=-1)$ defined by
\[
\phi_{l_{V,i}}^1(z)=1,\quad \phi_{l_{V,i}}^{2}(z\cdot j)=(-1)^{l_{V,i}}, \quad z\in \BC,
\]
and $\phi_{m_{V,i}}^2\ (m_{V,i}\in \BN)$ is the two-dimensional representation of $\CW_{\BR}$ with basis $u,v$ satisfying
\[
\begin{aligned}
\phi_{m_{V,i}}^2(z)u=u, &\quad  \phi_{m_{V,i}}^2(z\cdot j)u=(-1)^{m_{V,i}}v,\\
\phi_{m_{V,i}}^2(z)v=v, &\quad  \phi_{m_{V,i}}^2(z\cdot j)v=u.
\end{aligned}
\] 
The adjoint $L$-function $L(s,\phi_V,\Ad)=L(s,\phi_V^{\ss}\otimes \phi_V^{\ss,\vee})$ is a product of factors
\[
L(s,\phi_V,|\cdot|^{s_{V,i}^1}\phi_{l_{V,i}}^1\otimes(|\cdot|^{s_{V,j}^1}\phi_{l_{V,j}}^1)^{\vee}), 
L(s,\phi_V,|\cdot|^{s_{V,i}^1}\phi_{l_{V,i}}^1\otimes (|\cdot|^{s_{V,j}^2}\phi_{m_{V,j}}^2)^{\vee}), 
\]
\[
L(s,\phi_V,|\cdot|^{s_{V,i}^2}\phi_{m_{V,i}}^2\otimes (|\cdot|^{s_{V,j}^1}\phi_{l_{V,j}}^1)^{\vee}), L(s,\phi_V,|\cdot|^{s_{V,i}^2}\phi_{l_{V,i}}^2\otimes (|\cdot|^{s_{V,j}^2}\phi_{m_{V,j}}^2)^{\vee}).
\]
From \cite{knapp1982classification}, we can compute the value of these $L$-functions and obtain that 
\begin{enumerate}
    \item $L(s,\phi_V,|\cdot|^{s_{V,i}^1}\phi_{l_{V,i}}^1\otimes(|\cdot|^{s_{V,j}^1}\phi_{l_{V,j}}^1)^{\vee})$  has a pole at  $s=1$ if and only if $\frac{1+s_{V,i}^1-s_{V,j}^1+\frac{1-(-1)^{l_{V,i}+l_{V,j}}}{2}}{2}$  is a non-positive integer. 
    \item $L(s,\phi_V,|\cdot|^{s_{V,i}^1}\phi_{m_{V,i}}^1\otimes (|\cdot|^{s_{V,j}^2}\phi_{m_{V,j}}^2)^{\vee})$ has a pole at  $s=1$ if and only if  ${1+s_{V,i}^1-s_{V,j}^2+\frac{m_{V,j}}{2}}$  is a non-positive integer.
        \item $L(s,\phi_V,|\cdot|^{s_{V,i}^2}\phi_{m_{V,i}}^2\otimes (|\cdot|^{s_{V,j}^1}\phi_{l_{V,j}}^1)^{\vee})$ has a pole at  $s=1$ if and only if ${1+s_{V,i}^2-s_{V,j}^1+\frac{m_{V,i}}{2}}$  is a non-positive integer.
    \item $L(s,\phi_V,|\cdot|^{s_{V,i}^2}\phi_{m_{V,i}}^2\otimes (|\cdot|^{s_{V,j}^2}\phi_{m_{V,j}}^2)^{\vee})$ has a pole at  $s=1$ if and only if  $1+s_{V,i}^2-s_{V,j}^2+\frac{m_{V,i}^{2}+m_{V,j}}{2}$ or $1+s_{V,i}^2-s_{V,j}^2+\frac{|m_{V,i}-m_{V,j}|}{2}$  is a non-positive integer.
\end{enumerate}
\begin{lem}\label{lem: genericity}
An parameter $\phi_V$ with semisimplification $\phi_V^{\ss}$ in (\ref{equ: semisimplification}) is generic if and only if none of 
\[\begin{aligned}
&\frac{1+s_{V,i}^1-s_{V,j}^1+\frac{1-(-1)^{l_{V,i}+l_{V,j}}}{2}}{2},{1+s_{V,i}^1-s_{V,j}^2+\frac{m_{V,j}}{2}},
\\&{1+s_{V,i}^2-s_{V,j}^1+\frac{l_{V,i}}{2}},1+s_{V,i}^2-s_{V,j}^2+\frac{|m_{V,i}-m_{V,j}|}{2}
\end{aligned}
\]
is a non-positive integer.
\end{lem}

\subsubsection{Irreducibility criterions}

B. Speh and D. Vogan gave a sufficient condition for the irreducibility of limits of generalized principal series representations in \cite[Theorem 6.19]{speh1980reducibility}. We apply this result to prove the irreducibility of standard models for representations in generic packets. 

\begin{defin}\label{def: rtimes}
Given $\sigma_1\in \Pi(\GL_{n_1}), \dots,\sigma_r\in \Pi(\GL_{n_r})$ and $\pi_{V_0}\in \Pi(\SO(p,q))$. We denote by 
\[
\sigma_1\times \cdots\times\sigma_r\ltimes \pi_{V_0}
\]
the normalized parabolic induction
\[
\RI_{P_{n_1,\cdot,n_r,p+q}}^{\SO(p+n,q+n)}(\sigma_1\otimes \cdots \otimes\sigma_r\otimes\pi_{V_0})\in \Pi(\SO(p+n,q+n)),\quad n=n_1+\cdots+n_r.
\]
\end{defin}

\begin{lem}\label{lem: irreducible criterion}
Fix a generic parameter $\phi_{V}=\phi^{\GL}_V\oplus\phi_{V_0}\oplus (\phi^{\GL}_V)^{\vee}$ of $\SO(p,q) (p>q)$, for $\sigma\in \Pi_{\phi^{\GL}_V}$ and $\pi_{V_0}\in \Pi^{\Vogan}_{\phi_{V_0}}$, 
 the representation $\sigma\ltimes \pi_{V_0}$ is irreducible.
\end{lem}
\begin{proof}
From \cite[Theorem 14.2]{knapp1982classification}, we may write the tempered representation $\pi_{V_0}$ as a parabolic induction from a limit of discrete series representations. Then we can express $\sigma\ltimes \pi_{V_0}$ as
\begin{equation}\label{equ: irr check induction}
\sigma_{1}\times\cdots\times \sigma_l\ltimes \pi_{V_0'}\quad \sigma_i\in \Pi(\GL_{n_{V,i}})
\end{equation}
where $\pi_{V_0}'\in \Pi(\SO(V_0'))$ is a limit of discrete series representation and 
\[
\sigma_i=|\cdot|^{s_{V,i}^1}\sgn^i\text{ or }  \sigma_i=|\det|^{s_{V,i}^2}D_{m_{V,i}}.
\]

Following \cite[Theorem 6.19]{speh1980reducibility}, it suffices to check the following conditions
\begin{num}
\item
\label{equ: irreducible criterion} For every root $\alpha$ such that
\[
n_{\alpha}=\langle \alpha,\nu\rangle/\langle \alpha,\alpha \rangle \in \BZ,
\]\begin{enumerate}
   \item if $\alpha$ is a complex root ($\alpha\neq -\theta\alpha$), then $\langle \alpha,\nu\rangle\langle \theta\alpha,\nu\rangle\geqslant 0$;
    \item if $\alpha$ is a real root ($\alpha=-\theta\alpha$),  then
    \[(-1)^{n_{\alpha}+1}=\epsilon_{\alpha}\cdot\lambda(m_{\alpha})\]
    Here $\lambda$ is the central character of $\sigma$, $m_{\alpha}$ is the image of $\rho_{\alpha}(-I_2)$ in $G$ for the embedding $\rho_{\alpha}:\SL_2(\BR)\to G(\BR)$ determined by $\alpha$ and $\epsilon_{\alpha}=-1$.
\end{enumerate}
\end{num}
Then we check them using Lemma \ref{lem: genericity}.
\begin{enumerate}
    \item For every complex root $\alpha$ such that $n_{\alpha}\in\BZ$
    \begin{enumerate}
        \item If $\alpha$ is a root of $\SO(p-q)$, then
        \[\langle \alpha,\nu\rangle\langle \theta\alpha,\nu\rangle = 0\]
        \item Otherwise, $\theta\alpha=\alpha$, then
\[\langle \alpha,\nu\rangle\langle \theta\alpha,\nu\rangle=\langle \alpha, \nu\rangle^2 \geqslant 0\]
    \end{enumerate}
    \item For every real root $\beta_{ab}=e_{a}-e_{b}$  such that $n_{\beta_{ab}}\in \BZ$.
    
    \begin{enumerate}
        \item If $E_{aa}$ is in the $\GL_1$-block $\GL_{n_{V,i}}$  and $E_{bb}$ is the in a $\GL_1$-block $\GL_{n_{V,j}}$ (in the inducing datum in (\ref{equ: irr check induction})), then $n_{\beta_{ab}}=\frac{s_{V,i}^1-s_{V,j}^1}{2}$ is an integer, and both
        \[\frac{1+s_{V,i}^1-s_{V,j}^1+\frac{1-(-1)^{l_{V,i}+l_{V,j}}}{2}}{2},\frac{1+s_{V,j}^1-s_{V,i}^1+\frac{1-(-1)^{l_{V,i}+l_{V,j}}}{2}}{2}\]
        are not non-positive integers. If
        $l_{V,i}+l_{V,j}$ is odd, the sum is equal to $2$, then  $s_{V,i}^1=s_{V,j}^1$ or $s_{V,i}^1-s_{V,j}^1$ is odd. If $l_{V,i}+l_{V,j}$ is even, the sum is equal to $3/2$, then $s_{V,i}^1-s_{V,j}^1$ is even.
        \item If $E_{aa}$ is in the $\GL_1$-block $\GL_{n_{V,i}}$  and $E_{bb}$ is the in a $\GL_2$-block $\GL_{n_{V,j}}$,
        Lemma \ref{lem: genericity} implies 
        \[
s_{V,j}^2-\frac{l_{V,j}}{2}\leqslant s_{V,i}^1\leqslant s_{V,j}^2+\frac{l_{V,j}}{2}
        \]
        \item If $E_{aa}$ is in the $\GL_2$-block $\GL_{n_{V,i}}$  and $E_{bb}$ is the in a $\GL_2$-block $\GL_{n_{V,j}}$,  we may assume $l_{V,j}\geqslant l_{V,i}$,
        Lemma \ref{lem: genericity} implies 
        \[
 s_{V,j}^2-\frac{l_{V,j}}{2}\leqslant
    s_{V,i}^2-\frac{l_{V,i}}{2} \leqslant s_{V,i}^2+\frac{l_{V,i}}{2}\leqslant s_{V,j}^2+\frac{l_{V,j}}{2}
        \]
    \end{enumerate}
Therefore, we checked case (b)(c)  following an understanding of the parity condition in \cite[Theorem 2]{prasad2017reducible}. For case (a), the parity holds unless $l_{V,i}+l_{V,j}$ is odd and $s_{V,i}^1=s_{V,j}^1$. In this situation 
\[
|\cdot|^{s_{V,i}^1}\sgn^{l_{V,i}}\times |\cdot|^{s_{V,j}^1}\sgn^{l_{V,j}}=|\cdot|^{s_{V,i}^1}\sgn^{l_{V,i}}(1\times \sgn)\quad 
\]
And $1\times \sgn$ is the limit of discrete series representation with parameter $\phi_0^2$ which can be treated in cases (b)(c).
\end{enumerate}
\end{proof}

\subsubsection{Representations in generic packets}
 The classification of representations of $\CW_{\BR}$(\cite{knapp1982classification}) shows the following factorization into irreducible representations
\begin{equation}\label{equ: decomposition of varphi}
\phi_V^{\ss}=\phi^{\GL}_V\oplus\phi_{V_0}\oplus 
(\phi^{\GL}_{V})^{\vee}
\end{equation}
where $\phi_{V_0}$ is tempered  and 
\[
\phi_V^{\GL}=\bigoplus_{i=1}^{l_{V}}|\cdot|^{s_i}\phi_{V,i}^{\GL}  \quad\text{ where }\Re(s_{i})>0 \text{ for } 1\leqslant  i\leqslant l_{V}
\]
for discrete series parameter $\phi_{V,i}$ (i.e. the image of $\phi_{V,i}$ is bounded and does not lie in any proper Levi).

It is straightforward that $\phi_V^{\GL}$ is unpaired. Let $n_{V,i}=\dim \phi_{V,i}^{\GL}$, $n_V=\dim \phi_V^{\GL}$ and $\sigma_{V,i}$ be the unique representation of $\GL_{n}$ in the $L$-packet $\Pi_{\phi_{V,i}^{\GL}}(\GL_{n_{V,i}})$, then 
\begin{equation}\label{equ: sigmav}
\Pi_{\phi_{V}^{\GL}}(\GL_{n_V})=\{\sigma_V\}\quad\text{ where }\sigma_V=|\det|^{s_1}\sigma_{V,1}\times \cdots \times |\det|^{s_{l_{V}}}\sigma_{V,l_{V}}
\end{equation}
  By Lemma \ref{lem: irreducible criterion}, there is an injective map
  \begin{equation}\label{equ: main isom}
  \begin{aligned}
\Pi^{\Vogan}_{\phi_{V_0}}&\to\Pi^{\Vogan}_{\phi_V}\\
\pi_{V_0}&\mapsto \sigma_V\ltimes\pi_{V_0} 
  \end{aligned}
  \end{equation}
Since $\phi_V^{\GL}$ is unpaired, $|\CS_{\phi_{V_0}}|=|\CS_{\phi_V}|$ and thus $|\Pi^{\Vogan}_{\phi_{V_0}}|=|\Pi^{\Vogan}_{\phi_V}|$. This implies that the above map is an isomorphism and we have the following result.
\begin{pro}\label{pro: real ind}
For a generic $L$-parameter $\phi_V=\phi_V^{\GL}\oplus\phi_{V_0}\oplus 
(\phi_V^{\GL})^{\vee}$,  every representation $\pi_V$ in  $\Pi_{\phi_V}^{\Vogan}$ can be expressed as $\pi_V=\sigma_V\ltimes \pi_{V_0}$
where $\pi_{V_0}\in\Pi_{\phi_{V_0}}^{\Vogan}$ 
and $\sigma_V$ given in (\ref{equ: sigmav}). 
\end{pro}

Proposition \ref{pro: real ind} shows that representations in the generic packets are in the following form.
\begin{equation}\label{equ: para ind form}
\pi_V=\sigma_V\ltimes \pi_{V_0},\quad \sigma_V=|\det|^{s_{V,1}}\sigma_{V,1}\times \cdots\times |\det|^{s_{V,r}}\sigma_{V,r},
\end{equation}
where $\Re(s_{V,1})\geqslant \Re(s_{V,2})\geqslant \cdots \geqslant \Re(s_{V,r})>0$, $\pi_{V_0}\in \Pi_{\temp}^{\mathrm{irr}}(\SO(V_0))$. And $\sigma_{V,i}=\sgn^{l_{V,i}}$ for $l_{V,i}=0,1$ or $\sigma_{V_i}=D_{m_{V,i}}$ for $m_i\in \BN_+$. 


For $\pi_V$ in the form of \ref{equ: para ind form}, we define the following notions.

\begin{defin}
We parametrize the infinitesimal character of $\pi_V$ with the \textit{Harish-Chandra parameter} for $\pi_V$ in (\ref{equ: para ind form}) is defined as
\[
\nu=(\nu_1,\dots,\nu_r,\nu_{\pi_{V_0}})
\]
where $\nu_{\pi_{V_0}}$ is the Harish-Chandra parameter of the tempered representation $\pi_{V_0}$, $\nu_i=s_i$ when $\rho_{V,i}=\sgn^{l_{V,i}}$, and $\nu_i=(s_{V,i}+\frac{m_{V,i}}{2},s_{V,i}-\frac{m_{V,i}}{2})$ when $\rho_{V,i}=D_{m_{V,i}}$.
\end{defin}
\begin{defin}\label{def: leading} We define \textit{leading index} of $\pi_V$ as the largest number among $\Re(s_{V,i})$.  We denote it by $\mathrm{LI}(\pi_V)$. %
\end{defin}

\section{Proof for the real case}\label{section: real}
In this section, we complete the proof of the local Gross-Prasad conjecture (Conjecture \ref{conj: GP in introduction}) over the real field based on the tempered cases. More specifically, following the approach in \cite{moeglin2012conjecture}, we prove a multiplicity formula for the reduction to the tempered cases and conclude the conjecture with the tempered cases proved in \cite{chen2022local}. 

The proof uses the idea of Mackey's theory. Let $\mathsf G$ be a reductive group over $\BR$, $\mathsf H$ is a closed subgroup of $\mathsf G$ and $\mathsf P$ is a parabolic subgroup of $\mathsf G$ with Levi decomposition $\mathsf P=\mathsf M\mathsf N$. We denote by $G=\mathsf G(\BR)$, $H=\mathsf H(\BR)$ and $P=\mathsf P(\BR)$. For a representation $\sigma$ of $M=\mathsf M(\BR)$, we study the space $\Hom_{H}(\Ind_P^G(\sigma),1_H)$ by analyzing the double coset $P\bs G/H$. Since $P\bs G$ is compact, the smooth induction $\Ind_P^G(\sigma)$ is equal to the Schwartz induction in the sense of \cite{du1991representations}. In order to use the analytic tools established in \cite{du1991representations} and \cite{chen2020schwartz}, we work within the category of almost linear Nash groups (\cite[Definition 1.1]{sun2015almost}) and consider the category of Nash manifolds (\cite[Definition 2.1]{sun2015almost}), with the possible action of certain almost linear Nash groups.
In particular, for a linear algebraic group $\mathsf G$ over $\BR$, $\mathsf G(\BR)$ can be treated as an almost linear Nash group.

Let $G$ be an almost linear Nash group. We denote by  $\mathcal{SF}(G)$ the categories of smooth Fr\'echet $G$-representations of moderate growth. We denote by $\mathcal{CW}(G)$ the subcategory of $\mathcal{SF}(G)$ consisting of representations with admissible $(\Fg_{\BC},K)$-modules, that is, the category of Casselman-Wallach representations of $G$. We use $\mathrm{Irr}(G)$ to denote the set of irreducible  Casselman-Wallach representations of $G$.

Our main result in this section is the following theorem.
\begin{thm}(Multiplicity formula)\label{thm: reductive}
Let $V,W$ be quadratic spaces with decompositions
$V=V_0\perp (X_V+X_V^{\vee})$, $W=W_0\perp (X_W+X_W^{\vee})$.
Let $\pi_{V_0}\in \Irr(\SO(V_0))$, $\pi_{W_0}\in \Irr(\SO(W_0))$ and $\sigma_V\in \mathcal{CW}(\SO(V)),\sigma_W\in \mathcal{CW}(\SO(W))$ such that
\begin{equation}\label{equ: condition of sigma}
\sigma_V=|\det|^{s_{V,1}}\sigma_{V,1}\times \cdots \times |\det|^{s_{V,r_V}}\sigma_{V,r_V},\quad \sigma_W=|\det|^{s_{W,1}}\sigma_{W,1}\times \cdots \times |\det|^{s_{W,r_W}}\sigma_{W,r_W}
\end{equation}
for $\Re(s_{V,i}),\Re(s_{W,i})>0$ and tempered representations $\sigma_{V,i}\in \Irr(\GL_{n_{V,i}}(F))$ ($i=1,\dots, r_V$), $\sigma_{W,i}\in \Irr(\GL_{n_{W,i}}(F)$ ($j=1,\dots,r_W$). Here $n_{V,i},n_{W,i}$ are integers such that $\sum_{i=1}^{r_V}n_{V,i}=\dim X_V$ and $\sum_{i=1}^{r_W}n_{W,i}=\dim X_W$. Then we have 
\[
m((\sigma_{V}\ltimes \pi_{V_0})\boxtimes (\sigma_{W}\ltimes \pi_{W_0}))=m(\pi_{V_0}\boxtimes\pi_{W_0}).
\]
\end{thm}
It worth mentioning that, in the above theorem, the representations $\sigma_{V}\ltimes \pi_{V_0}$ and $\sigma_{W}\ltimes \pi_{W_0}$ can be reducible. The reducible case of  the multiplicity formula is actually necessary when it is applied in \cite{chen2024arithmetic}. In this article, to complete the proof for the real case of the Gross-Prasad conjecture, we only use the formula when both $\sigma_{V}\ltimes \pi_{V_0}$ and $\sigma_{W}\ltimes \pi_{W_0}$ are irreducible.

\begin{proof}[Proof for Theorem \ref{thm: real} from Theorem \ref{thm: reductive}]
Given generic parameters $\phi_{V},\phi_W$,   from Proposition \ref{pro: real ind},  we can express the parameters as
\begin{equation}\label{equ: dec of para}
\phi_V=\phi_V^{\GL}+\phi_{V_0}+(\phi_V^{\GL})^{\vee},\quad \phi_W=\phi_W^{\GL}+\phi_{W_0}+(\phi_W^{\GL})^{\vee}
\end{equation}
such that $\phi_V^{\GL}$ has no self-dual subrepresentation.

 Let $\sigma_{V}$ be the unique representation in $\Pi_{\phi_V^{\GL}}^{\Vogan}$ and $\sigma_W$ be the unique representation in $\Pi_{\phi_W^{\GL}}^{\Vogan}$. For every $\pi_V\boxtimes\pi_W\in \Pi_{\phi_V\times \phi_W}^{\Vogan}$, there exists $\pi_{V_0}\boxtimes \pi_{W_0}\in \Pi^{\Vogan}_{\phi_V\times \phi_W}$  such that
\[
\pi_V=\sigma_V\ltimes \pi_{V_0},\quad \pi_W=\sigma_W\ltimes \pi_{W_0}.
\]
Therefore, the map \[
    \begin{aligned}
        \Pi_{\phi_{V_0}}^{\Vogan}&\to \Pi_{\phi_V}^{\Vogan}&\quad\Pi_{\phi_{W_0}}^{\Vogan}&\to \Pi_{\phi_W}^{\Vogan}\\
        \pi_{V_0}&\mapsto \sigma_{V}\ltimes \pi_{V_0} &\pi_{W_0}&\mapsto \sigma_W\ltimes \pi_{W_0}.
    \end{aligned}
    \]
    is an isomorphism. Hence, we can identify  the component group $\CS_{\phi_{V_0}\times \phi_{W_0}}$ with $\CS_{\phi_V\times \phi_W}$. Under this identification, it can be easily verified that for $\phi_V,\phi_W,\phi_{V_0},\phi_{W_0}$, we have \[
    \eta_{\phi_{V_0}\times \phi_{W_0}}=\eta_{\phi_{V}\times \phi_{W}}.\]  Theorem \ref{thm: reductive} reduces Conjecture \ref{conj: GP in introduction} for $\phi_V,\phi_W$ to that for $\phi_{V_0},\phi_{W_0}$, which is a tempered case  proved in \cite{luothesis}\cite{chen2022local}.
\end{proof} 

 Following \cite{moeglin2012conjecture}, there are three steps in our proof for Theorem \ref{thm: reductive}: reduction to basic cases, the first inequalities, and the second inequalities.
 
 A relevant pair $(W,V)$ is called \textit{basic} if $\dim V=\dim W+1$.
For a general relevant pair  $(W,V)$ with decomposition $V=W\perp Z\perp D$, we let $D^+$ be the anisotropic line with the opposite signature to $D$. We set $Z^+=Z\perp (D+D^+)$ and set $(V,W^+)=(V,Z^+\oplus W)$ and we call $(V,W^+)$ the \textit{basic relevant pair associate to} $(W,V)$.

\begin{defin}
Let $s_1,s_2,\dots,s_{r+1}$ be complex numbers. We say the $(r+1)$-tuple $\underline{s}=(s_1,\dots,s_{r+1})$ are \textit{in general position}, if $\underline{s} \in
\BC^{t+1}$ does not lie in the set of zeros of countably many polynomial functions on $\BC^{t+1}$.
\end{defin}
For the $(r+1)$-tuple $\udl{s}=(s_1,\dots,s_{r+1})$, we denote by $\sigma_{\underline{s}}$ the spherical principal series representation $|\cdot|^{s_1}\times \cdots \times |\cdot|^{s_{r+1}}$.
\begin{lem}(Reduction to basic cases)\label{lem: reduction to codim 1}
For  every $\pi_V\in \Irr(\SO(V))$ and $\pi_W\in \Irr(\SO(W))$, we have
\[
m(\pi_V\boxtimes\pi_W)=m((\sigma_{\underline{s}}\ltimes\pi_W)\boxtimes\pi_V)
\]
for $\underline{s}=(s_1,\dots,s_{r+1})\in \BC^{r+1}$ in general position.
\end{lem}
With this lemma, we find such a spherical principal series $\sigma_{\udl{s}}$ and reduce Theorem \ref{thm: reductive} to the case for a relevant pair $(V, W\oplus Z^+)$ and representations $\sigma_{\udl{s}}\ltimes \pi_W,\pi_V$ that can be expressed in the parabolic induction form as in (\ref{equ: para ind form}), which is a basic case.

\begin{pro}(Basic case of the multiplicity formula)\label{prop: multiformulatemperneeded}
Given a basic relevant pair $(W, V)$, let $\pi_V\in \mathcal{CW}(\SO(V))$ and $\pi_W\in \mathcal{CW}(\SO(W))$ as in Theorem \ref{thm: reductive}, we have
\[
m(\pi_V\boxtimes \pi_W)=m(\pi_{V_0}\boxtimes\pi_{W_0}).
\]
\end{pro}
The inequalities $m(\pi_V\boxtimes \pi_W)\geqslant m(\pi_{V_0}\boxtimes \pi_{W_0})$ and $m(\pi_V\boxtimes \pi_W)\leqslant m(\pi_{V_0}\boxtimes \pi_{W_0})$ are called "the first inequality" and "the second inequality" in \cite{moeglin2012conjecture}. Using a similar approach as \cite{moeglin2012conjecture}, we prove the first inequality using mathematical induction with the following lemma as the building block (Section \ref{section: multiplicity formula}).
 

\begin{lem}\label{lem: basic}
Let $\pi_V$ be a representation in a generic packet and $\pi_W\in \Irr(\SO(W))$.
\begin{enumerate}
\item When $\dim V=\dim W+1$ and $\Re(s)\geqslant \mathrm{LI}(\pi_V)$, we have
\[
m(\pi_V\boxtimes\pi_W)\geqslant m((|\cdot|^s\sgn^m\ltimes \pi_W)\boxtimes\pi_V).
\]

\item When $\dim V=\dim W+3$ and $\Re(s)\geqslant \mathrm{LI}(\pi_V)$, we have
\[
m(\pi_V\boxtimes(|\cdot|^{s+\frac{m}{2}}\sgn^{m+1}\ltimes \pi_W)) \geqslant m((|\det|^sD_m\ltimes \pi_W)\boxtimes\pi_V),
\]
where $D_m$ is the Langlands quotient of the induction $|\cdot|^{-\frac{m}{2}}\times |\cdot|^{\frac{m}{2}}\sgn^{m+1}$.
\end{enumerate}
\end{lem} 

The second inequality holds in a more general setup.
 \begin{lem}\label{lem: basic second}
For $\pi_V\in \mathcal{CW}(\SO(V))$, $\pi_W\in \mathcal{CW}(\SO(W))$ and $\sigma_{X^+}$ is a generic representation in $\GL(X^+)$, we have
\[
m(\pi_V\boxtimes\pi_W)\leqslant m((\sigma_{X^+}\ltimes \pi_W)\boxtimes\pi_V)
\]
 \end{lem}

We prove one inequality of Lemma \ref{lem: reduction to codim 1} and Lemma \ref{lem: basic} in Section \ref{section: multiplicity formula} and prove the other inequality of Lemma \ref{lem: reduction to codim 1} and Lemma \ref{lem: basic second} in Section \ref{section: harmonic analysis}. It is worth mentioning that Lemma \ref{lem: reduction to codim 1} can also be proved with Schwartz homology as in \cite{xue1bessel}.
 
\subsection{Some functors and vanishing theorems}
In this section, we review some analytic tools established in \cite{du1991representations} and \cite{chen2020schwartz} to study certain Fréchet spaces of moderate growth.

\subsubsection{Schwartz induction}
Let $G$ be an almost linear Nash group.
\begin{pro}\label{pro: tensor is exact}
For $\pi\in \mathcal{CW}(G)$, the projective tensor product  $\cdot\wh{\otimes} \pi$ is an exact functor in $\mathcal{SF}(G)$.
\end{pro}
\begin{proof}
From \cite{bernstein2014smooth}, the underlying Fr\'echet space of $\pi$ is nuclear and the proposition follows from \cite[Lemma A.3]{casselman2000bruhat}.
\end{proof}

    Let $H$ be a Nash subgroup of $G$ and $\pi_H\in \mathcal{SF}(H)$. We denote by $H\bs (G\times \pi_H)$ the vector bundle over $H\bs G$ obtained by $G\times \pi_H$ quotient by left $H$-action 
    \begin{equation}\label{equ: left action def}
    h.(g,v)=(h\cdot g,\pi_H(h).v) \quad \text{for }h\in H,\quad g\in G \text{ and }v\in \pi_H.
    \end{equation} Moreover, this vector bundle is tempered. We define the \textit{Schwartz induction} as the functor \[
    \Ind^{\CS,G}_P:\mathcal{SF}(H)\to \mathcal{SF}(G)
    \]
\[
\pi_H\mapsto\Gamma^{\mathcal{S}}(H\backslash G,\pi_H),
\]
where $\Gamma^{\mathcal{S}}(H\backslash G,\pi_H)$ stands for the space of Schwartz sections over the tempered vector bundle $H\backslash (G\times \pi_H)$.
In particular, when $G$ is reductive, for a parabolic subgroup $P$ of $G$, $P\bs G$ is compact, so the Schwartz induction $\Ind_P^{{\CS,G}}$ coincides with the smooth induction, and we denote by $\RI_P^G$ the normalized induction $\Ind_P^{\CS,G}(\delta_P^{1/2}\cdot)$, where $\delta_{P}$ is the modular characters of $P$. We will  use the following properties of Schwartz inductions. 
\begin{pro}\label{pro: Schwartz induction}
\begin{enumerate}
    \item(\cite[Propositon 7.1]{chen2020schwartz}) $\Ind^{\CS,G}_H$ is an exact functor from $\mathcal{SF}(H)\to \mathcal{SF}(G)$.
    \item (\cite[Proposition 7.2]{chen2020schwartz}) For a closed subgroup $H'$ of $H$, we have \[\Ind_H^{\CS,G}\circ\Ind_{H'}^{\CS,H}=\Ind_{H'}^{\CS,G}.\]
    \item (\cite[Proposition  7.4]{chen2020schwartz}\cite{bernstein2014smooth}) For $\pi_G\in \mathcal{CW}(G)$ and $\pi_H\in\mathcal{SF}(H)$, then
    \[
    \Ind_{H}^{\CS,G}(\pi_H\wh{\otimes}\pi_G|_H)=\Ind_H^{\CS,G}(\pi_H)\wh{\otimes}\pi_G.
    \]
\end{enumerate}
\end{pro}

\subsubsection{$\Hom$-functor}

For any category $\mathcal{C}$ and an object $M$. It is well-known that the $\Hom(-,M)$-functor is left exact and invariant under projective limit. We first apply this result to the category $\mathcal{SF}(G)$ and obtain the following result.
\begin{lem}\label{lem: hom pre 0}
\begin{enumerate}
    \item For an exact sequence 
    \[
    0\to \pi_1\to \pi_2\to \pi_3\to 0
    \]
    in $\SF(G)$, suppose $\Hom_G(\pi_1,1_{G})=\Hom_G(\pi_3,1_{G})=0$, then we have
    \[
    \Hom_G(\pi_2,1_{G})=0.
    \]
    \item For a directed set $I$  and projective system $(\pi_{\alpha},f_{\alpha\beta})_{\alpha,\beta\in I}$ in $\SF(G)$. For $I'\subset I$, suppose 
    \[
   \Hom_{G}(\pi_{\alpha},1_G)=0, \quad \text{for all } \alpha\in I',
    \]
    then we have 
    \[
    \Hom(\varprojlim_{i\in I}\pi_{\alpha},1_G)=0.
    \]
\end{enumerate}
\end{lem}

\begin{defin}\label{label: complete decreasing}
\begin{enumerate}
    \item For a countable directed set $I$ and a Fr\'echet space $V$, a set $\{V_k\}_{k\in I}$ of subspaces of $V$ is called a \textbf{complete decreasing filtration} of $\pi$ if  

\begin{enumerate}
\item For $i<j$, $V_j\subset V_i$, and we denote by  $f_{ji}$ the injective map. 
\item $\{V_i,f_{ji}\}_{i<j\in I}$ is a complete projective system, that is, 
\[
\varprojlim_{i\in I} V/V_i=V.
\]
\end{enumerate}
\item 
The \textbf{composition factors} of a complete decreasing filtration are 
\[
V_{\alpha}/V_{\alpha+}, \quad \alpha\in I,
\]
where $\alpha+$ is the successor of $\alpha$ in $I$.
\end{enumerate}
\end{defin}

\begin{cor}\label{cor: graded vanishing}
For an almost linear Nash group $G$, $\pi\in \mathcal{SF}(G)$ and a complete decreasing filtration $\{\pi_k\}_{k\in I}$ of $\pi$, suppose 
\[
\Hom_G(V_{\alpha}/V_{\alpha+},1_{G})=0,\quad \text{ for all }\alpha\in I,
\]
then we have
\[
\Hom_G(\pi,1_G)=0.
\]
\end{cor}
\begin{proof}
This corollary can be obtained from Lemma \ref{lem: hom pre 0} with the same argument as \cite{xue2020bessel}.
\end{proof}

The results \cite[Propositions 8.2,  8.3]{chen2020schwartz} generate a complete decreasing filtration that is helpful for distributional analysis. 
\begin{thm}\label{thm: CScdf}
    Let $\CX$ be a Nash manifold, $\CZ$ be a closed Nash manifold of $\CX$ and $\CU=\CX-\CZ$. There is a decreasing complete  filtration on $\Gamma^{\mathcal{S}}_{\mathcal{Z}}(\mathcal{X},\mathcal{E})$, denoted by $\Gamma^{\mathcal{S}}_{\mathcal{Z}}(\mathcal{X},\mathcal{E})_k$, 
whose composition factors are isomorphic to 
\begin{equation}
\Gamma^{\mathcal{S}}(\mathcal{Z},\operatorname{Sym}^{k} \mathcal{N}_{\mathcal{Z} / \mathcal{X}}^{\vee} \otimes \mathcal{E}|_{\mathcal{Z}}), \quad k=0,1, \dots,    
\end{equation}
where $\mathcal{N}_{\mathcal{Z}/\mathcal{X}}^{\vee}$ is the conormal bundle over $\mathcal{Z}$ (\cite[Section 6.1]{chen2020schwartz}).
\end{thm}
\subsubsection{Vanishing by infinitisimal characters}

\begin{defin}
For an infinitesimal character $\chi:\CZ(\CU(\Fg_{\BC}))\to \BC$, we denote by $\chi^{\vee}$ the infinitesimal character generated by the relation
\[
\chi^{\vee}(X)=\chi(-X),\quad X\in \Fg_{\BC}.
\]
\end{defin}
    
\begin{thm}\label{thm: vanishing by infinitisimal}
For representation $\pi_1,\pi_2$ of $G$ with infinitesimal characters $\chi_{\pi_1},\chi_{\pi_2}$, suppose $\chi_{\pi_1}\neq \chi_{\pi_2}^{\vee}$, then we have
\[
\Hom_G(\pi_1\wh{\otimes} \pi_2,1_{G})=0.
\]
\end{thm}
\begin{proof}
The existence of elements in $\Hom_G(\pi_1\wh{\otimes} \pi_2,1_{G})$ implies the existence of a homomorphism on $(\Fg_{\BC},K)$-modules. This contradicts the relation of infinitesimal characters.
\end{proof}
We apply the above theorem in the following setup:
\begin{cor}
\label{thm: vanishing 1}
Suppose $\pi_{V_0}\in \mathcal{SF}(\SO(V_0))$ and $\pi_{V}\in \Irr(\SO(V))$.  \[(\sigma_{\underline{s}}\ltimes\pi_{V_0})\wh{\otimes} \pi_V
\]
for $\sigma_{\underline{s}}=|\cdot|^{s_1}\times \cdots \times |\cdot|^{s_r}$ and $\underline{s}=(s_1,\dots,s_r)$ in general positions.    
\end{cor}

\subsubsection{Vanishing by leading index} 

\begin{defin}\label{def: A leading}
By Langlands classification, for every $\pi_V\in \Irr(\SO(V))$, we can express $\pi_V$ as the Langlands quotient of certain induction 
\begin{equation}\label{equ: Appendix std md}
|\det|^{s_1}\rho_1\times \cdots |\det|^{s_r}\rho_r\ltimes \pi_{V_0}
\end{equation}
for $\Re(s_1)\geqslant\cdots\geqslant\Re(a_r)>0$ and tempered representations $\rho_1,\dots,\rho_r,\pi_{V_0}$. We let the \textit{leading index for Langlands quotient} $\mathrm{LI}(\pi_V)=\Re(s_1)$. This definition is compatible with Definition \ref{def: leading} when the standard module (\ref{equ: Appendix std md}) is irreducible. In particular, the definitions are compatible when $\pi_V$ is in a generic packet.
\end{defin}

\begin{thm}\label{thm: CCZ}(\cite[Theorem A.1.1]{chen2023multiplicity})
If $\Re(s)>\mathrm{LI}(\pi_V)$, then 
\[
\Hom_{\Delta \SO(V)}((|\det|^{s}\rho\ltimes\pi_{V_0})\boxtimes \pi_V,1_{\Delta\SO(V)})=0,\quad \text{ for }\pi_{V_0}\in \mathcal{SF}(\SO(V_0)), \pi_V\in \Irr(\SO(V)). 
\]
\end{thm}
\subsection{The restriction of principal series to mirabolic subgroups}\label{section: mirabolic}
In this section, we study the graded structure of the restriction of certain principal series of $\GL_n$ to the mirabolic subgroup $R_{n-1,1}$ as in \cite[\S 5]{xue2020bessel}, that is, the subgroup of $\GL_n$ leaving $V_n/V_{n-1}$ invariant, where $V_n$ is the space of the standard representation of $\GL_n$ and $V_{n-1}$ is an $(n-1)$-dimensional subspace of $V_n$. These results will be used in the distributional analysis of the open orbit in Section \ref{section: multiplicity formula}.
\subsubsection{Graded structure of $|\cdot|^{-\frac{m}{2}}\times |\cdot|^{\frac{m}{2}}\sgn^{m+1}$}
By definition, the discrete series $D_m$ of $\GL_2(\BR)$ is the unique quotient of the induction $\pi_I=|\cdot|^{-\frac{m}{2}}\times |\cdot|^{\frac{m}{2}}\sgn^{m+1}$. We denote $\pi_F$ the unique subrepresentation of this induction $\pi_I$, then $\pi_m$ is an $m$-dimensional irreducible representation of $\GL_2(\BR)$. 
\begin{itemize}
\item Let $\mathsf B_2$ be the (upper-triangular) Borel subgroup of $\GL_2$ with Levi decomposition  $\mathsf B_2=\mathsf T_2 \mathsf N_2$. Let $K=\SO_2(\BR)$, $B_2=\mathsf B_2(\BR)$, $T_2=\mathsf T_2(\BR)$, $N_2=\mathsf N_2(\BR)$ and $R_{1,1}=\mathsf R_{1,1}(\BR)$.
\item We note 
\[
n_x=\left(\begin{array}{cc}
    1 & x \\
     & 1
\end{array}\right)\quad k_{\theta}=\left(\begin{array}{cc}
    \cos(\theta) & \sin(\theta) \\
    -\sin(\theta) & \cos(\theta)
\end{array}\right)\quad w_2=\left(\begin{array}{cc}
     & 1 \\
    1 & 
\end{array}\right),
\]
then $N_2=\{n_x:x\in \BR\}$ and $K=\{k_{\theta}:\theta\in [0,2\pi)\}$.
\item We note $\CX_2=B_2\bs \GL_2(\BR)$, $\mathcal{U}_2=B_{2}\backslash B_{2}w_2B_{2}\subset \mathcal{X}_2$ and   $\mathcal{Z}_2=B_{2}\backslash B_{2}$.
\item By definition,
\[
\pi_I=\Ind_{B_2}^{\CS,\GL_2(\BR)}(|\cdot|^{\frac{m+1}{2}}\otimes |\cdot|^{\frac{m-1}{2}}\sgn^{m+1}).
\]
We denote $\chi_1=|\cdot|^{-m+1}\sgn^{m+1}$ and $\chi_2=|\cdot|^{\frac{m-1}{2}}\sgn^{m+1}$, then
\[
\pi_I= \Ind_{B_2}^{\CS,\GL_2(\BR)}(\chi_1\chi_2\otimes \chi_2).
\]

\end{itemize}

\begin{lem}\label{weightfd}
\begin{enumerate}
    \item The representation $\pi_F$ is isomorphic to the $n$-dimensional $\GL_2(\BR)$-representation 
\[
\chi_1\chi_2(\det(\cdot))\Sym^{n-1}(\mathbb{C}^2),
\]
where $\BC^2$ is the standard representation of $\GL_2(\BR)$.
\item The restriction $\pi_F|_{R_{1,1}}$ has irreducible components 
\[
|\det(\cdot)|^k\sgn^{k}(\det(\cdot)), \text{ for }k=0,1,\dots,m-1.
\]
\end{enumerate} 
\end{lem}
\begin{proof}
Point (1) follows directly from \cite[\S 2.3]{godement1974notes} and Point (2) follows from direct computation based on (1).
\end{proof}
Using the left quotient in the sense of \eqref{equ: left action def}, we define
\[
\CE_2:=B_{2}\bs (\GL_2(\BR)\times \chi_1\chi_2\otimes \chi_2).
\]
Extension by zero gives a natural embedding of $R_{1,1}$-representations
\begin{equation}\label{naturalembedding}
   i_{UX}:\Gamma^{\mathcal{S}}(\mathcal{U}_2,\CE_2)\to \Gamma^{\mathcal{S}}(\mathcal{X}_2,\CE_2).
\end{equation}

\begin{lem}\label{infinitygradedpieces}
There is a complete decreasing filtration $\{\Gamma^{\mathcal{S}}_{\mathcal{Z}}(\mathcal{X}_2,\CE_2)_i\}_{i\in \BN}$ of submodules of  $\Gamma^{\mathcal{S}}(\mathcal{X}_2,\CE_2)/\Gamma^{\mathcal{S}}(\mathcal{U}_2,\CE_2)$  such that the composition factors
 are $R_{1,1}$-isomorphic to 
\[
\chi_1\chi_2(\det(\cdot))\sgn^k(\det(\cdot))|\det(\cdot)|^{k}\big|_{R_{1,1}},\quad k\in \BN.
\]
\end{lem}
\begin{proof}
This lemma follows from \cite[Propositions 8.2, 8.3]{chen2020schwartz}.
\end{proof}
We identify $\Gamma^{\mathcal{S}}(\mathcal{U}_2,\CE_2)$ as $\Ind_{\mathbb{R}^{\times}\times 1}^{\CS,R_{1,1}}(\chi_2)$ using the following equation.
\[
\begin{aligned}
\Gamma^{\mathcal{S}}(\mathcal{U}_2,\CE_2)&=\Gamma^{\mathcal{S}}(B_2\backslash B_2w_2B_2,\CE_2)\\
&=\Gamma^{\mathcal{S}}(T_2\backslash B_2,\chi_2 \otimes \chi_1\chi_2)\\
&=\Gamma^{\mathcal{S}}(\mathbb{R}^{\times}\times 1\backslash R_{1,1},\chi_2)=\Ind_{\mathbb{R}^{\times}\times 1}^{\CS,R_{1,1}}(\chi_2),
\end{aligned}
\]
and then define an $R_{1,1}$-homomorphism 
\[
T_d:\Ind_{\mathbb{R}^{\times}\times 1}^{\CS,R_{1,1}}(\chi_2)\to \pi_D
\]
by  compositing  the embedding (\ref{naturalembedding}) and the quotient map $\pi_{I}$ to $\pi_{F}$, that is, 
\[
T_{d}:\Ind_{\mathbb{R}^{\times}\times 1}^{\CS,R_{1,1}}(\chi_2)= \Gamma^{\mathcal{S}}(\mathcal{U}_2,\CE_2) \hookrightarrow\Gamma^{\mathcal{S}}(\mathcal{X}_2,\CE_2)=\pi_{I}\rightarrow \pi_{I}/\pi_{F}=\pi_{D}.
\]
\begin{lem}
 The homomorphism $T_d$ is injective.
\end{lem}
\begin{proof}
Suppose $T_d$ is not injective then there exist $\widetilde{f}\in \Gamma^{\mathcal{S}}(\mathcal{U},\chi_1\chi_2 \otimes \chi_2)$ such that its extension by zero $\widetilde{f}_G$  in $\pi_{I}$ is contained in $\pi_F$. 

On the one hand, $f(x)=\widetilde{f}(w_2n_x)$ is a Schwartz function.  For $\theta\in (0,\pi)$, we can compute $\widetilde{f}$ with the decomposition
\[
k_{\theta}=\left(\begin{array}{cc}
    1/\sin(\theta) & \cos(\theta) \\
     & \sin(\theta)
\end{array}\right)w_2\left(\begin{array}{cc}
    1 & -\mathrm{cot}(\theta) \\
     & 1
\end{array}\right).
\]
Then we have 
\[
\widetilde{f}_G(k_{\theta})=\widetilde{f}(k_{\theta})=\chi_{1}\chi_2(1/\sin(\theta))\chi_2(\sin(\theta))f(-\mathrm{cot}(\theta))=o(\theta^l), \quad \text{ for every }l>0.
\]
Then $(\frac{d}{d\theta})^l\widetilde{f}_G(k_{\theta})|_{\theta=0}=0$ for every positive integer $l$.

On the other hand, from \cite[Section 2.3]{godement1974notes},  $\pi_F$ is generated by the functions
\[
\phi_{-m+1},\phi_{-m+3},\dots,\phi_{m-1},
\]
where
$\phi_l\left(n_x\cdot t(a,b)\cdot k_{\theta}\right)=\chi_1\chi_2(a)\chi_2(b)e^{il\theta}$.

Then $\widetilde{f}_G\in \pi_F$ is a linear combination of $\phi_k$, that is, there is a nonzero $n$-tuple $(\lambda_1,\dots,\lambda_n)\in \mathbb{C}^n$ such that $\widetilde{f}_G=\sum_{k=1}^{n}\lambda_k\phi_{2k-n-1}$.
Then we have
\[
\left.(\frac{d}{d\theta})^l\widetilde{f}_G(k_{\theta})\right|_{\theta=0}=\sum_{k=0}^{n-1}\lambda_k((2k-n-1)i)^l,
\]
Hence, there exists $l$ such that $(\frac{d}{d\theta})^l(\widetilde{f}_G(k_{\theta}))|_{\theta=0}\neq 0$,  which leads to a contradiction. Therefore, the $R_{1,1}$-homomorphism $T_d$ is injective. 
\end{proof}

\begin{pro}\label{pro: GL2 graded piece}
$\mathrm{Coker}(T_d)
$ has a  decreasing complete filtration $\Gamma^{\mathcal{S}}_{\mathcal{Z}}(\mathcal{X}_2,\CE_2)_k$ with composition factors isomorphic to 
\begin{equation}
    |\det(\cdot)|^{k+\frac{m-1}{2}}\sgn(\cdot)^k\big|_{R_{1,1}},\text{ for } k=1,2,\dots.
\end{equation}
\end{pro}
\begin{proof}
From Lemma \ref{infinitygradedpieces}, $    \Gamma^{\mathcal{S}}_{\mathcal{Z}}(\mathcal{X}_2,\CE_2)=\pi_{I}/\Gamma^{\mathcal{S}}(\mathcal{U}_2,\CE_2)$ has a  decreasing complete filtration $\Gamma^{\mathcal{S}}_{\mathcal{Z}}(\mathcal{X}_2,\CE_2)_k$ with composition factors isomorphic to 
\begin{equation}
    |\det(\cdot)|^{k}\sgn(\cdot)^k\chi_1\chi_2(\det(\cdot))\big|_{R_{1,1}},\text{ for } k=0,1,\dots.
\end{equation}
From Lemma \ref{weightfd},  the finite-dimensional representation $\pi_F$ in $\pi_I$  has a $R_{1,1}$-composition factors with irreducible pieces
\[|\det(\cdot)|^{k}\sgn^k(\det(\cdot))\chi_1\chi_2(\det(\cdot))\big|_{R_{1,1}}, \text{ for }k=0,1,\dots,m-1.\] 
 Then the projection $\pi_I \to \pi_I/i_{UX}(\Gamma^{\mathcal{S}}(\mathcal{U}_2,\CE_2))$ gives an isomorphism between $\pi_F$ and  $\overline{\pi}_{{F}}=\Gamma^{\mathcal{S}}_{\mathcal{Z}}(\mathcal{X}_2,\CE_2)/\Gamma^{\mathcal{S}}_{\mathcal{Z}}(\mathcal{X}_2,\CE_2)_n$. Therefore, we have 
\[
    \Gamma^{\mathcal{S}}_{\mathcal{Z}_2}(\mathcal{X}_2,\CE_2)=\pi_F\oplus \Gamma^{\mathcal{S}}_{\mathcal{Z}_2}(\mathcal{X}_2,\CE_2)_m.
    \]
    Therefore,
    \[
\mathrm{Coker}(T_d)=\pi_D/i_{UX}(\Gamma^{\mathcal{S}}(\mathcal{U}_2,\CE_2))=(\pi_{I}/\Gamma^{\mathcal{S}}(\mathcal{U}_2,\CE_2))/\pi_F= \Gamma^{\mathcal{S}}_{\mathcal{Z}_2}(\mathcal{X}_2,\CE_2)_m,\]
and thus has a  decreasing complete filtration with composition factors isomorphic to 
\begin{equation*}\label{infinitedecreasing}
   \sigma_k=|\det(\cdot)|^{k}\sgn^{k}(\det(\cdot))\chi_2(\det(\cdot))\big|_{R_{1,1}}= |\det(\cdot)|^{k+\frac{m-1}{2}}\sgn(\cdot)^k\big|_{R_{1,1}},\text{ for } k=1,2,\dots.
\end{equation*}
\end{proof}

\subsubsection{Graded structure of spherical principal series} Let $(s_1,\dots,s_{r+1})\in \BC^{r+1}$, and we denote $\sigma_{X^+}=|\cdot|^{s_1}\times \cdots \times |\cdot|^{s_{r+1}}$, which is a spherical principal series. The computation in \cite[Section 5.1]{xue2020bessel} for the restriction of spherical principal series representations to the mirabolic subgroup $R_{r,1}$ can be generalized over the real field verbatim and we can obtain   a proposition parallel to \cite[Proposition 5.1]{xue2020bessel}.

Following \cite[\S 5]{xue2020bessel},  we denote by $Q_{a,b,c}$  the intersection of the parabolic subgroup $P_{a,b,c}$ associated to the partition $(a,b,c)$ in $\GL_{a+b+c}$ and the mirabolic subgroup $R_{a+b+c-1}$. We let the "Levi part" $L_{a,b,c}$ of $Q_{a,b,c}$ to be the image of $\GL_a\times \GL_b\times R_{c-1,1}$ diagonally embedded into $\GL_{a,b,c}$. Then $Q_{a,b,c}=L_{a,b,c}U_{a,b,c}$ for the unipotent group associated to the partition $(a,b,c)$.
\begin{pro}\label{prop: graded pieces}
When restricted to $R_{r,1}$, the representation $\sigma_{X^+}$ has a subrepresentation $\Ind_{N_{r+1}}^{\CS,R_{r,1}}(\psi_{r+1}^{-1})$. Moreover, the quotient $\sigma_{X^+}/\Ind_{N_{r+1}}^{\CS,\GL_{r+1}}(\psi_{r+1}^{-1})$ admits an $R_{r,1}$-stable complete filtration whose composition factors have the shape 
\[
\Ind_{Q_{a,b,c}}^{\CS,R_{r,1}}(\tau_a\boxtimes \tau_b\boxtimes \tau_c)
\]
where $a+b+c=t+1$, $a+b\neq 0$ and the tensor $\tau_a\boxtimes \tau_b\boxtimes \tau_c$ is regarded as a $Q_{a,b,c}$ representation by trivially extended by $N_{a,b,c}$.

\begin{enumerate}
    \item $\tau_a=\Ind_{B_a}^{\CS,\GL_a(\BR)}(\sgn^{m_1}|\cdot|^{s_{i_1}+k_1}\boxtimes \cdots \boxtimes \sgn^{m_{a}}|\cdot|^{s_{i_a}+k_a})$
    where $1\leqslant i_1, \dots,i_a\leqslant t+1$ are integers, $l_1,\dots,l_a\in\BZ$ and $k_1,\dots,k_a\in\frac{1}{2}\BZ$;

    \item $\tau_b=\tau_b'\otimes \rho$ where $\tau_b'$ is a representation of the same form as $\tau_a$ and $\rho$ is a finite-dimensional representation of $\GL_b(\BR)$;
    \item $\tau_c=\Ind_{N_c}^{\CS,R_{c-1,1}}(\psi_{c}^{-1})$.
\end{enumerate}
\end{pro}

\subsection{Multiplicity formula: first inequality}\label{section: multiplicity formula}
In this section, we prove Lemma \ref{lem: basic} and one inequality of Lemma \ref{lem: reduction to codim 1}. More precisely, in the setting of Theorem \ref{thm: reductive}, we prove the inequality 
\begin{equation*}
m(\pi_V\boxtimes\pi_W)\geqslant m((|\det|^s\sigma_{X^+}\ltimes\pi_W)\boxtimes\pi_V)
\end{equation*}
for a basic relevant pair $(W^+,V)$ when 
\begin{enumerate}
\item $\sigma_{X^+}=\sgn^l$ and $s\geqslant \mathrm{LI}(\pi_V)$, or
    \item  $\sigma_{X^+}=\sigma_{\udl{s}}$ for $\udl{s}$ in general positions.
\end{enumerate}
With a similar approach, we show that
\[
m(\pi_V\boxtimes(|\cdot|^{s+\frac{m}{2}}\sgn^{m+1}\ltimes \pi_W)) \geqslant m((|\det|^s\sigma_{X^+}\ltimes\pi_W)\boxtimes\pi_V)
\]
when $\sigma_{X^+}=D_m$ and $s\geqslant \mathrm{LI}(\pi_V)$.

For a relevant pair $(W, V)$ and we let  $(V, W^+)$ be the associated basic relevant pair with the decomposition $W^+=W\perp (X^+\oplus Y^+)$. We denote by $(\mathsf G^+,\mathsf H^+,\xi^+)$ the Gross-Prasad triple associated to $(V,W^+)$. 

Let $P_{X^+}$ be the parabolic subgroup of $\SO(W^+)$ stabilizing  $X^+$. For $\sigma_{X^+}\in \SF(\GL(X^+))$ and $\pi_{W}\in \SF(\SO(W))$. From Definition \ref{def: rtimes}, 
 \[
\sigma_{X^+}\ltimes \pi_{W}=\Ind_{P_{X^+}}^{\CS,G}(|\det|^s\sigma_{X^+}\ltimes\pi_W)=\Gamma^{S}(P_{X^+}\bs \SO(W^+),\mathcal{E})
\]
where
\begin{equation}\label{equ: def of CE}   
\CE=\CE_{\sigma_{X^+},\pi_W}=P_{X^+}\bs (\SO(W^+)\times (\delta_{P_{X^+}}^{1/2}|\det|^s\sigma_{X^+}\boxtimes\pi_W)).
\end{equation}

We first study the structure of the  right-$\SO(V)$-orbits of $\CX=P_{W^+}\bs \SO(W^+)$.
\begin{enumerate}
    \item 
When $\dim W^+>2(r+1)$, $\CX$ consists of all $k$-dimensional totally isotropic subspaces of $V$. When $\dim W^+=2(r+1)$, there are exactly two maximal totally isotropic spaces and $\CX$ is exactly one of them.
\item 
When $\dim W^+>2(r+1)$, there is an open $\SO(V)$-orbit $\CU$ consisting of $(r+1)$-dimensional totally isotropic spaces that is not contained in $V$. Its complement $\CZ$ is the space of $(r+1)$-dimensional totally 
isotropic spaces contained in $V$. When $\dim V=2(r+1)$ and  $X^+.g_0\subset V$ for some $g_0\in \SO(W^+)$, $\CZ$ has two orbits and both of them are singletons, more precisely, $[X^+.g_0]$ and $[X^+.g_0g]$ for any $g\in \RO(V)\bs \SO(V)$; when $\dim V=2(r+1)$ and if $X^+.g_0\not\subseteq V$ for all $g_0\in \SO(W^+)$, $\CZ$ is empty; otherwise, $\CZ$ has just one orbit. 
\end{enumerate}

Then we can conclude that
\begin{lem}\label{lem: orbit}
\begin{enumerate}
    \item $\CZ$ is empty, when $\dim W^+=2(r+1)$, or $\dim V=2(r+1)$ and $X^+.g_0\not\subseteq V$ for all $g_0\in \SO(W^+)$;
    \item $\CZ$ has two $\SO(V)$-orbits, when $\dim V\neq 2(r+1)$;
    \item $\CZ$ has a single $\SO(V)$-orbit, when $\dim V=2(r+1)$ and $X^+.g_0\subseteq V$ for some $g_0\in \SO(W^+)$. 
\end{enumerate}
\end{lem}

 Let $\Gamma^{\CS}_{\CZ}(\CX,\CE)=\Gamma^{S}(\mathcal{X},\mathcal{E})/\Gamma^{S}(\mathcal{U},\mathcal{E})$. From Proposition \ref{pro: tensor is exact}, there is a short exact sequence
\begin{equation}\label{equ: exact sequence of Schwartz sections}
0\to \Gamma^{S}(\mathcal{U},\mathcal{E})\boxtimes \pi_V\to\Gamma^{S}(\mathcal{X},\mathcal{E})\boxtimes\pi_V\to \Gamma^{\CS}_{\CZ}(\CX,\CE)\boxtimes\pi_V\to 0
\end{equation}

Hence, we have the short exact sequence
\begin{equation}\label{equ: exact sequence of orbits}
\begin{aligned}
0\to \Hom_{H^+}(\Gamma^{\CS}_{\CZ}(\CX,\CE)\boxtimes\pi_V,1_{H^+})\to \Hom_{H^+}(\Gamma^{S}(\mathcal{X},\mathcal{E})\boxtimes\pi_V,1_{H^+})&\\
\to \Hom_{H^+}(\Gamma^{S}(\mathcal{U},\mathcal{E})\boxtimes \pi_V,1_{H^+})
\end{aligned}
\end{equation}
When $\Hom_{H^+}(\Gamma^{\CS}_{\CZ}(\CX,\CE)\boxtimes\pi_V,1_{H^+})=0$, we have 
\[
m((\sigma_{X^+}\ltimes \pi_W)\boxtimes \pi_V)\leqslant \dim \Hom_{H^+}(\Gamma^{S}(\mathcal{U},\mathcal{E})\boxtimes \pi_V,1_{H^+}).
\]
We first analyze the closed orbits on $\CZ$ to prove 
\[
\Hom_{H^+}(\Gamma^{\CS}_{\CZ}(\CX,\CE)\boxtimes\pi_V,1_{H^+})=0
\] and then analyze the open orbit $\CU$ to prove 
\[
\dim \Hom_{H^+}(\Gamma^{S}(\mathcal{U},\mathcal{E})\boxtimes \pi_V,1_{H^+})\leqslant m(\pi_V\boxtimes \pi_W),
\]
under the given conditions.

\subsubsection{Closed orbits}

When $\CZ$ is nonempty. Let $\gamma\in \SO(W^+)$ be a representative of an orbit of $\CZ$ such that $X^+.\gamma=X'$ where $X'$ is a totally isotropic subspace of $V$ satisfying $\dim X^+=\dim X'$. Then the stabilizer group $S_{\gamma}$ at $[X]$ is 
 equal to $\gamma^{-1}P_{W^+}\gamma\cap \SO(V)$, which a parabolic subgroup of $\SO(V)$ with Levi decomposition $S_{\gamma}=M_{\gamma}N_{\gamma}$ and the Levi subgroup $M_{\gamma}=\GL(X')\times \SO(V_0)$. The cotangent bundles and their fibers at $[X']$ are
\[
\begin{aligned}    
T^*_{\CZ}&=\SO(V)\times_{S_{\gamma}} S_{\gamma}^{\perp},\quad \mathrm{Fib}_{[X']}(T^*_{\CZ})=S_{\gamma}^{\perp}\\
T^*_{\CX}&=\SO(W^+)\times_{P_{W^+}}P_{W^+}^{\perp},\quad \mathrm{Fib}_{[X']}(T^*_{\CX})=P_{W^+}^{\perp}
\end{aligned}
\] 
 and $S_{\gamma}$ acts by adjoint action. Then the fiber of the conormal bundle at $[X']$
\[
\mathrm{Fib}_{[X']}(\CN_{\CZ/\CX}^{\vee})=\mathrm{Fib}_{[X']}(T^*_{\CX})/\mathrm{Fib}_{[X']}(T^*_{\CZ})=P_{W^+}^{\perp}/S_{\gamma}^{\perp},
\]
which is $\dim(X')$-dimensional. The  $\SO(V_0)$ and $N_{\gamma}$ act trivially and $\GL(X')$ acts as the standard representations.
 Then 
 \[
 \begin{aligned}
&\Gamma^{\mathcal{S}}(\SO(V).[X],\operatorname{Sym}^{k} \mathcal{N}_{\mathcal{Z} / \mathcal{X}}^{\vee} \otimes \mathcal{E}|_{\mathcal{Z}})\\
=&\Ind_{S_{\gamma}}^{\CS,\SO(V)}(\mathrm{Fib}_{[X]}(\Sym^{k}\CN_{\CZ/\CX}\otimes \CE|_{\CZ}))\\
=&\RI_{S_{\gamma}}^{\SO(V)}((|\det(\cdot)|^{s+\frac{1}{2}}\sigma_{X^+}\otimes \Sym^k\rho_{X'}^{\mathrm{std}}) \boxtimes ({ }^{\gamma}\pi_W|_{\SO(V_0)}))
 \end{aligned}
 \]
 
 Therefore,
\begin{equation}\label{equ: section on closed variety}
\Gamma^{\mathcal{S}}(\mathcal{Z},\operatorname{Sym}^{k} \mathcal{N}_{\mathcal{Z} / \mathcal{X}}^{\vee} \otimes \mathcal{E}|_{\mathcal{Z}})=(\RI_{S_{\gamma}}^{\SO(V)}((|\det(\cdot)|^{s+\frac{1}{2}}\sigma_{X^+}\otimes \Sym^k\rho_{X'}^{\mathrm{std}}) \boxtimes ({ }^{\gamma}\pi_W|_{\SO(V_0)})))^{\oplus c}
\end{equation}

Here $\rho_{X'}^{\mathrm{std}}$  is the standard representation of $\GL(X')$ 
and $c$ is the number of $\SO(V)$-orbits in $\CZ$.

\begin{pro}\label{pro: closed vanishing}We have
 \[\Hom_{H^+}(\Gamma^{S}_{\mathcal{Z}}(\mathcal{X}, \mathcal{E})\boxtimes\pi_V,1_{H^+})=0\]
 under any of the following conditions:
 \begin{enumerate}
 \item $\sigma_{X^+}=\sgn^l$ ($l=0,1$) or $\sigma_{X^+}=D_m$ ($m\in \BN_+$), and $s\geqslant \mathrm{LI}(\pi_V)$, or
     \item $\sigma_{X^+}=\sigma_{\udl{s}}\in \BC^{r}$ and $\udl{s}$ is in general position.
 \end{enumerate}
\end{pro}
 \begin{proof}
By \eqref{equ: section on closed variety}, we have
\[
\Gamma^{\mathcal{S}}(\mathcal{Z},\Sym^k\mathcal{N}_{\mathcal{Z}/\mathcal{X}}^{\vee}\otimes \mathcal{E}|_{\mathcal{Z}})\boxtimes\pi_V=(\RI_{S_{\gamma}}^{\SO(V)}(|\det(\cdot)|^{s+\frac{1}{2}}(\sigma_{X^+}\otimes \Sym^k\rho^{\mathrm{std}}_{X'}) \boxtimes ({ }^{\gamma}\pi_W|_{\SO(V_0)})))^{\oplus c}{\boxtimes} \pi_V
\]
\begin{itemize}
\item When $\sigma_{X^+}=\sgn^m$, we have $\sigma_{X^+}\otimes \Sym^k\rho=|\det|^k\sgn^m$. When $\Re(s)\geqslant \mathrm{LI}(\pi_V)$, we have $s+\frac{1}{2}+k>\mathrm{LI}(\pi_V)$, from Theorem \ref{thm: CCZ}, we have
\[
   \Hom_{H^+} (\Gamma^{\mathcal{S}}(\mathcal{Z},\Sym^k\mathcal{N}_{\mathcal{Z}/\mathcal{X}}^{\vee}\otimes \mathcal{E}|_{\mathcal{Z}})\boxtimes\pi_V,1_{H^+})=0.
    \]
\item When $\sigma_{X^+}=D_m$, by computation with the base of $D_{m+2a}$ in \cite[\S 2.3]{godement1974notes}, we have 
\[
\sigma_{X^+}\otimes \Sym^k\rho=\bigoplus_{a=0}^{k}D_{m+2a}.
\] 
When $\Re(s)\geqslant \mathrm{LI}(\pi_V)$, we have $s+\frac{1}{2}>\mathrm{LI}(\pi_V)$, from Theorem \ref{thm: CCZ}, we have
\[
   \Hom_{H^+} (\Gamma^{\mathcal{S}}(\mathcal{Z},\Sym^k\mathcal{N}_{\mathcal{Z}/\mathcal{X}}^{\vee}\otimes \mathcal{E}|_{\mathcal{Z}})\boxtimes\pi_V,1_{H^+})=0.
    \]
    \item When $\sigma_{X^+}=|\cdot|^{s_1}\times\cdots \times|\cdot|^{s_r}$, from \cite[Corollary 5.6]{kostant1975tensor}, the Harish-Chandra parameters  of the infinitesimal character of  $\sigma_{X^+}\otimes \Sym^k\rho$ is $[(s_1+a_1,\dots,s_{r+1}+a_{r+1})]$, where $a_i$ are non-negative integers. From Corollary \ref{thm: vanishing 1}, we have
    \[
   \Hom_{H^+} (\Gamma^{\mathcal{S}}(\mathcal{Z},\Sym^k\mathcal{N}_{\mathcal{Z}/\mathcal{X}}^{\vee}\otimes \mathcal{E}|_{\mathcal{Z}})\boxtimes\pi_V,1_{H^+})=0
    \]
    for $\udl{s}\in \BC^{r+1}$ in general positions.

    From Corollary \ref{cor: graded vanishing}, we can conclude that, under the conditions given in the proposition, we have 
    \[ 
    \Hom_{H^+}(\Gamma^{S}_{\mathcal{Z}}(\mathcal{X}, \mathcal{E})\boxtimes\pi_V,1_{H^+})=0.
    \]
\end{itemize}
Hence, from \eqref{equ: exact sequence of orbits}, we have 
\[
\dim\Hom_{H^+}(\Gamma^{S}(\mathcal{U}, \mathcal{E})\boxtimes\pi_V,1_{H^+}) \leqslant\dim\Hom_{H^+}(\Gamma^{S}(\mathcal{U}, \mathcal{E})\boxtimes\pi_V,1_{H^+}).
\]
\end{proof}

\subsubsection{The open orbit}
Then we study $\Gamma^{S}(\mathcal{U}, \mathcal{E})$ and show that 
\[
\dim\Hom_{H^+}(\Gamma^{S}(\mathcal{U}, \mathcal{E})\boxtimes\pi_V,1_{H^+})
\] is less than or equal to $m(\pi_V\boxtimes \pi_W)$ under the given conditions.

We introduce the following notations just for this section
 \begin{itemize}
    \item Let $d=\dim V$, $r=\frac{\dim  V-\dim W-1}{2}$ We can compute the modular character 
    \[
    \delta_{P_{X^+}}((m\times g_{W})\ltimes n)=|\det(m)|^{d-1-r},\quad m\in \GL(X^+), g_W\in \SO(W), n\in N.
    \]
    \item Let $N_{r+1}$ be the unipotent subgroup of $\GL_{r+1}(\BR)$ consisting of upper-triangular unipotent matrices, and let $R_{r,1}$ be the mirabolic subgroup of $\GL_{r+1}$. We denote by $N_{r,1}$ the unipotent radical of $R_{r,1}$.
\item We define a generic character $\pi_{r+1}$ of $N_{r+1}$ by letting 
    \[
    \psi_{r+1}(n)=\psi(\sum_{i=1}^{r+1}n_{i,i+1}),
    \]
    where $n_{i,j}$ is the entry of matrix $n$ at $i$-th row and $j$-th column.
\end{itemize}    

    Recall the decomposition $V=W\perp D\perp Z$ in Section \ref{section: Gross-Prasad triple}.  Let $X=X^+\cap Z$ and we have $X$ is totally isotropic and $\dim X=\dim X^+-1$. Let  $N$ be the unipotent radical of the parabolic subgroup $P_X$ of $\SO(V)$ stabilizing $X$. We define $N'_{V}$ the subgroup of $N$ stabilizing $D$, then $H=(N_{r+1}\times \Delta \SO(W))\ltimes N_{V}'$. 

From Frobenius reciprocity, we have 
\begin{equation}\label{equ: FR}
\Hom_{H}(\xi^{-1}\otimes (\pi_V\boxtimes\pi_W),1_{H})=\Hom_{H^+}(\Ind_{H}^{\CS,H^+}(\xi^{-1}\otimes(\pi_V\boxtimes\pi_W),1_{H^+}).
\end{equation}
By definition, the dimension of the left-hand side of \eqref{equ: FR} is equal to $m(\pi_V\boxtimes \pi_W)$. The right-hand side of \eqref{equ: FR} can be expressed as
\begin{equation}\label{equ: open open expression}
\begin{aligned}
\Ind_{H}^{\CS,H^+}(\xi^{-1}\otimes(\pi_V\boxtimes \pi_W))&=\Ind_{(N_{r+1}\times \Delta\SO(W))\ltimes N_{V}'}^{\CS,H^+}(\xi^{-1}\otimes (\pi_V\boxtimes \pi_W))\\
&=\Ind_{(R_{r,1}\times \Delta \SO(W))\ltimes N_{V}'}^{\CS,H^+}(\Ind_{N_{r+1}}^{\CS,R_{r,1}}(\psi_{r+1}^{-1})|_{R_{r,1}}\boxtimes\pi_W\boxtimes\pi_V).
\end{aligned}
\end{equation}

Recall that the open orbit  $\CU=P_{W^+}\bs P_{W^+}\SO(V)=(P_{W^+}\cap \SO(V))\bs \SO(V)$ and the stabilizer group
can be decomposed as
\begin{equation}\label{equ: definition Rr1}
P_{W^+}\cap \SO(V)=(\GL(X)\times 1\times \SO(W))\ltimes N=\SO({W})\ltimes(R_{r,1}\ltimes N_{V}').
\end{equation}
By definition, we have
\begin{equation}\label{equ: open expression with tensor}
\begin{aligned}  
\Gamma^{S}_{\mathcal{Z}}(\mathcal{U}, \mathcal{E})\boxtimes \pi_V
=&\Ind_{P_{W^+}\cap \SO(V)}^{\CS,\SO(V)}(|\det|^{\frac{d-1-r}{2}}\sigma_{X^+}\otimes \pi_W|_{P_{W^+}\cap \SO(V)})\boxtimes \pi_V\\
=&\Ind_{P_{W^+}\cap \SO(V)}^{\CS,\SO(V)}(|\det|^{\frac{d-1-r}{2}}\sigma_{X^+}|_{R_{r,1}}\boxtimes\pi_W)\boxtimes\pi_V\\
=&\Ind_{(R_{r,1}\times \Delta \SO(W))\ltimes N_{V}'}^{\CS,H^+}(|\det|^{\frac{d-1-r}{2}}\sigma_{X^+}|_{R_{r,1}}\boxtimes\pi_W\boxtimes\pi_V)
\end{aligned}
\end{equation}
\begin{itemize}
    \item When $r=0$ and $\sigma_{X^+}=\sgn^l$, we have
    \[
    \Ind^{\CS,R_{r,1}}_{N_{r+1}}(\psi_{r+1}^{-1})|_{R_{r,1}}=|\det|^{\frac{d-1-r}{2}}\sigma_{X^+}|_{R_{r,1}},
    \]
so the right-hand side of (\ref{equ: open expression with tensor}) and (\ref{equ: open open expression}) are the same. Hence, we have
\[
m(\pi_V\boxtimes \pi_W)=\dim \Hom_{H^+}(\Gamma^{S}(\mathcal{U}, \mathcal{E})\boxtimes\pi_V,1_{H^+}).
\]

\item When $\sigma_{X^+}=|\cdot|^{s_1}\times \cdots\times|\cdot|^{s_{r+1}}$ for $(s_1,\dots, s_{r+1})\in \BC^n$, from Proposition \ref{prop: graded pieces}, there is a $R_{r,1}$-equivariant embedding
\begin{equation}\label{equ: embedding}
\Ind^{\CS,R_{r,1}}_{N_{r+1}}(\psi_{r+1}^{-1})\hookrightarrow |\det|^{\frac{d-1-r}{2}}\sigma_{X^+}.
\end{equation}

Applying the quotient of (\ref{equ: open expression with tensor}) and (\ref{equ: open open expression}), we obtain that
\begin{equation}\label{equ:}
\begin{aligned}
&\Gamma^{S}_{\mathcal{Z}}(\mathcal{U}, \mathcal{E})\boxtimes \pi_V/\Ind_H^{\CS,H^+}(\xi^{-1}\otimes (\pi_V\boxtimes \pi_W))\\
=&\Ind_{(R_{r,1}\times \Delta \SO(W))\ltimes N_{V}'}^{\CS,H^+}((|\det|^{\frac{d-1-r}{2}}\sigma_{X^+}|_{R_{r,1}}/\Ind_{N_{r+1}}^{\CS,R_{r,1}}\psi_{r+1}^{-1})\boxtimes\pi_W\boxtimes\pi_V).
\end{aligned}
\end{equation}

 Therefore, to conclude the inequality 
\[
\dim\Hom_{H^+}(\Gamma^{S}(\mathcal{U}, \mathcal{E})\boxtimes\pi_V,1_{H^+})\leqslant m(\pi_V\boxtimes \pi_W),
\] it suffices to prove that 
\begin{equation}\label{equ: Hom of graded}
\Hom_{H^+}(\Ind_{(R_{r,1}\times \Delta \SO(W))\ltimes N_{V}'}^{\CS,H^+}((|\det|^{\frac{d-1-r}{2}}\sigma_{X^+}|_{R_{r,1}}/\Ind_{N_{r+1}}^{\CS,R_{r,1}}(\psi_{r+1}^{-1})|_{R_{r,1}})\boxtimes\pi_W\boxtimes\pi_V),1_{H^+})=0.
\end{equation}

 Using Proposition \ref{prop: graded pieces}, from the exactness of Schwartz induction (Proposition \ref{pro: Schwartz induction}) and projective tensor product (Proposition \ref{pro: tensor is exact}), we obtain that the quotient 
\[
\Ind_{(R_{r,1}\times \Delta \SO(W))\ltimes N_{V}'}^{\CS,H^+}((|\det|^{\frac{d-1-r}{2}}\sigma_{X^+}|_{R_{r,1}}/\Ind_{N_{r+1}}^{\CS,R_{r,1}}\psi_{r+1}^{-1})\boxtimes\pi_W\boxtimes\pi_V)
\]
has composition factors 
\begin{equation}\label{equ: graded pieces open}
 \Ind_{(R_{r,1}\times \Delta \SO(W))\ltimes N_{V}'}^{\CS,H^+}(\Ind^{\CS,R_{r,1}}_{Q_{a,b,c}}({\tau_a\boxtimes \tau_b\boxtimes \tau_c})\boxtimes\pi_W),
\end{equation}
where $Q_{a,b,c}=P_{a,b,c}\cap R_{r,1}$ and $\tau_a,\tau_b,\tau_c$ defineed in Proposition \ref{prop: graded pieces}. Since (\ref{equ: graded pieces open}) can be expressed as the parabolic induction
\[
(|\det|^{-\frac{d-1-r+c}{2}}(\tau_a\boxtimes \tau_b))\ltimes \Ind_{(R_{c-1,1}\times \SO(W))\ltimes N_{W^+,c}}^{\CS,\SO(W\oplus D\oplus X_c)}(\xi_c^{-1}\otimes \pi_W),
\]
 based on Corollary \ref{cor: graded vanishing} and the fact that $a+b\geqslant 1$, the $Hom$-space in (\ref{equ: graded pieces open}) vanishes for $(s_1,\dots,s_{r+1})\in \BC^n$ in general position.

\item When $r=1$ and $\sigma_{X^+}=D_l$, instead of \eqref{equ: open open expression}, we use the following 
\begin{equation}\label{equ: open open expression discrete}
\begin{aligned}
\Ind_{\Delta \SO(W\oplus \BR)}^{\CS,H^+}((|\cdot|^{s+\frac{m}{2}}\sgn^{m+1}\ltimes \pi_W)\boxtimes \pi_V)
=\Ind_{(R_{1,1}\times \Delta \SO(W))\ltimes N_{V}'}^{\CS,H^+}(\Ind_{\BR^{\times}\times 1}^{\CS,R_{1,1}}(\chi_2)\boxtimes \pi_W\boxtimes \pi_V),
\end{aligned}
\end{equation}

From Section \ref{section: mirabolic}, there is an injection 
$T_d:\Ind_{\mathbb{R}^{\times}\times 1}^{\CS,R_{1,1}}(\chi_2)\hookrightarrow D_m$, and it induces an injection
\[
\Ind_{\mathbb{R}^{\times}\times 1}^{\CS,R_{1,1}}(|\cdot|^s\chi_2)\hookrightarrow |\det|^sD_m.
\]
Applying the quotient of (\ref{equ: open expression with tensor}) and (\ref{equ: open open expression discrete}), we obtain that
\begin{equation}
\begin{aligned}
&\Gamma^{S}_{\mathcal{Z}}(\mathcal{U}, \mathcal{E})\boxtimes \pi_V/\Ind_{\Delta \SO(W\oplus \BR)}^{\CS,H^+}((|\cdot|^{s+\frac{m}{2}}\sgn^{m+1}\ltimes \pi_W)\boxtimes \pi_V)\\
=&\Ind_{(R_{r,1}\times \Delta \SO(W))\ltimes N_{V}'}^{\CS,H^+}((|\det|^{\frac{d-2}{2}}\sigma_{X^+}|_{R_{1,1}}/\Ind_{N_{2}}^{\CS,R_{1,1}}(\psi_{2}^{-1}))\boxtimes\pi_W\boxtimes\pi_V).
\end{aligned}
\end{equation}

From Proposition \ref{pro: GL2 graded piece}, the quotient 
    $|\det|^s\sigma_{X^+}|_{R_{1,1}}/\Ind_{\BR^{\times}\times 1}^{\CS,R_{1,1}}(|\cdot|^s\chi_2)|_{R_{1,1}}$ has composition factors
   \[
   \sigma_k:=|\det(\cdot)|^{s+k+\frac{m-1}{2}}\sgn(\cdot)^k|_{R_{1,1}},\quad k=1,2,\dots
   \]
     From the exactness of Schwartz induction (Proposition \ref{pro: Schwartz induction}) and projective tensor product (Proposition \ref{pro: tensor is exact}), there is a decreasing complete filtration of
    \[
    \Ind_{(R_{r,1}\times \Delta \SO(W))\ltimes N_{V}'}^{\CS,H^+}((|\det|^{s+\frac{d-2}{2}}\sigma_{X^+}|_{R_{1,1}}/\Ind_{N_{2}}^{\CS,R_{1,1}}(\psi_{2}^{-1})|_{R_{1,1}})\boxtimes\pi_W\boxtimes\pi_V)
    \]
    with composition factors
    \[
    \Ind_{(R_{1,1}\times \Delta \SO(W))\ltimes N_{V}'}^{\CS,H^+}(\sigma_k\boxtimes\pi_W\boxtimes\pi_V).
    \]
Notice that
\[
\Ind_{(R_{1,1}\times  \Delta \SO(W))\ltimes N_{V}'}^{\CS,H^+}(\sigma_k\boxtimes \pi_W\boxtimes \pi_V)= (|\cdot|^{s+\frac{m}{2}+k}\sgn^m\ltimes \Ind_{\SO(W)}^{\CS,\SO(W\oplus \BR)}(\pi_W))\boxtimes \pi_V.
\]

Since we have assumed that $\Re(s)\geqslant \mathrm{LI}(\pi_V)$ and $k$ is a positive integer, we have 
  \[
  s+\frac{m}{2}+k>\mathrm{LI}(\pi_V).
  \]
   
   Then, from Theorem \ref{thm: CCZ}, we have
   \[
   \Hom_{H^+}((|\cdot|^{}\sgn^m\ltimes \Ind_{\SO(W)}^{\CS,\SO(W\oplus \BR)}(\pi_W))\boxtimes \pi_V,1_{H^+})=0,\quad k=1,2,\dots
   \]
   From Corollary \ref{cor: graded vanishing}, this implies
   \[
   \Hom_{H^+}(\Gamma^{S}_{\mathcal{Z}}(\mathcal{U}, \mathcal{E})\boxtimes \pi_V/\Ind_{\Delta \SO(W\oplus \BR)}^{\CS,H^+}((|\cdot|^{s+\frac{m}{2}}\sgn^{m+1}\ltimes \pi_W)\boxtimes \pi_V),1_{H^+})=0.
   \]
   Hence, Lemma \ref{lem: hom pre 0}, we have
   \[
   m(\pi_V\boxtimes(|\cdot|^{s+\frac{m}{2}}\sgn^{m+1}\ltimes \pi_W)) \geqslant \dim \Hom_{H^+}(\Gamma^{S}(\mathcal{U}, \mathcal{E})\boxtimes\pi_V,1_{H^+}).
   \]
\end{itemize}

\subsubsection{Proof for the first inequality}
Then we make use of Lemma \ref{lem: basic} to prove the inequality
\begin{equation}\label{equ: first inequ proof}
m(\pi_V\boxtimes \pi_W)\leqslant m(\pi_{V_0}\boxtimes \pi_{W_0})
\end{equation}
in Proposition \ref{prop: multiformulatemperneeded}.

\begin{proof}
We express $\pi_V=\sigma_V\ltimes \pi_{V_0},\pi_W=\sigma_W\ltimes\pi_{W_0}$ in the form of (\ref{equ: para ind form}) and prove the inequality by mathematical induction on
\[
N(\sigma_V,\sigma_W)=\sum_{\Re(s_{V,i})\neq 0}n_{V,i}+\sum_{\Re(s_{W,i})\neq 0}n_{W,i}.
\]
Here $s_{V,i},s_{W,i},n_{V,i},n_{W,i}$ are defined as in (\ref{equ: para ind form}).

If $N(\sigma_V,\sigma_W)=0$, both $\pi_V$ and $\pi_W$ are tempered, then the inequality follows from the Conjecture \ref{conj: GP in introduction} for tempered parameters, which was proved in \cite{luothesis}\cite{chen2022local}.

In other cases, we may assume 
\begin{eqnarray*}
 \Re(s_{V,1})\geqslant \Re(s_{V,2})\geqslant \cdots \geqslant \Re(s_{V,l})> 0,
\quad \Re(s_{W,1})\geqslant \Re(s_{W,2})\geqslant \cdots \geqslant \Re(s_{W,l})> 0.
\end{eqnarray*}

Suppose the proposition holds when $N(\sigma_V,\sigma_W)\leqslant k$, then when $N(\sigma_V,\sigma_W)=k+1$, we consider the following cases.
\begin{enumerate}
\item[Case 1:] If $l_V\neq 0$ and $\Re(s_{V,1})\geqslant \Re(s_{W,1})$, then let  $\widetilde{\sigma}_V=|\det(\cdot)|^{s_{V,2}}\sigma_{V,2}\times \cdots\times |\det(\cdot)|^{s_{V,l}}\sigma_{V,l}$.
\begin{enumerate}
\item If $n_{V,1}=1$, from Lemma \ref{lem: basic}(1) we have
\[
m((\sigma_V\ltimes \pi_{V_0})\boxtimes(\sigma_W\ltimes \pi_{W_0}))\leqslant m((\sigma_W\ltimes \pi_{W_0})\boxtimes(\widetilde{\sigma}_V\ltimes \pi_{V_0})),
\]
\item If $n_{V,1}=2$, let $\widehat{\sigma}_V=|\cdot|^{s_{V,1}+\frac{m_{V,1}}{2}}\sgn^{m_{V,1}+1}\ltimes \widetilde{\sigma}_V$ , and from Lemma \ref{lem: basic}(2), we have
\[
m((\sigma_V\ltimes \pi_{V_0})\boxtimes(\sigma_W\ltimes \pi_{W_0}))\leqslant m((\sigma_W\ltimes \pi_{W_0})\boxtimes(\widehat{\sigma}_V\ltimes \pi_{V_0})),
\]
\end{enumerate}

Since $N(\widetilde{\sigma}_V,\sigma_W),N(\widehat{\sigma}_V,\sigma_W)\leqslant N(\sigma_V,\sigma_W)-1=k$, we have
\[
m((\sigma_W\ltimes \pi_{W_0})\boxtimes(\widetilde{\sigma}_V\ltimes \pi_{V_0}))\leqslant m(\pi_{V_0}\boxtimes \pi_{W_0}),\quad m((\sigma_W\ltimes \pi_{W_0})\boxtimes(\widehat{\sigma}_V\ltimes \pi_{V_0}))\leqslant m(\pi_{V_0}\boxtimes \pi_{W_0})
\]
Therefore, we have
    \[
    m((\sigma_V\ltimes \pi_{V_0})\boxtimes(\sigma_W\ltimes \pi_{W_0}))\leqslant m(\pi_{V_0}\boxtimes\pi_{W_0}),
    \]
\item[Case 2:] If $l_V=0$ or $\Re(s_{V,1})<\Re(s_{W,1})$, then we switch the order of $V,W$ to reduce to Case 1. More explicitly, we take $\sigma_W^{(s')}=|\cdot|^{s'}\times \sigma_{W_0}$.  
    there is a $s'\in i\mathbb{R}$, such that 
    \[
    m((\sigma_V\ltimes \pi_{V_0})\boxtimes(\sigma_W\ltimes\pi_{W_0}))=m((\sigma_{W}^{(s')}\ltimes \pi_{W_0})\boxtimes(\sigma_V\ltimes\pi_{V_0}))
    \]
    From \cite[Theorem 1.1]{speh1980reducibility} and Langlands classification, we may assume $\sigma_W^{(s')}\ltimes \pi_{W_0}$ is irreducible. Then the pair $(\sigma_W^{(s')},\sigma_V)$ belongs to Case 1 and $N(\sigma_W^{(s')},\sigma_V)=N(\sigma_V,\sigma_W)=k+1$, so 
    \[
    m((\sigma_W^{(s')}\ltimes \pi_{W_0})\boxtimes(\sigma_V\ltimes\pi_{V_0}))\leqslant m( \pi_{V_0}\boxtimes\pi_{W_0}).
    \]Therefore, we have
\[
    m((\sigma_V\ltimes \pi_{V_0})\boxtimes(\sigma_W\ltimes \pi_{W_0}))\leqslant m(\pi_{V_0}\boxtimes\pi_{W_0}).
    \]
\end{enumerate}
Then the proposition follows from mathematical induction on $N(\sigma_V,\sigma_W)$.
\end{proof}
\subsection{Multiplicity formula: second inequality}\label{section: harmonic analysis}
In this section, we complete the proof for the second inequality of Proposition \ref{prop: multiformulatemperneeded}.
\subsubsection{A construction}
We prove Lemma \ref{lem: basic second} by construction. Recall that, for a relevant pair $(W, V)$, we can construct a basic relevant pair $(V, W^+)$ by taking $W^+=W{\perp} (X^+\oplus Y^+)$ for certain totally isotropic spaces $X^+$ and $Y^+$. Let $\mathsf G^+=\SO(W^+)\times \SO(V)$, $\mathsf H^+=\Delta\SO(V)$, $\mathsf P^+$ is the parabolic subgroup $\mathsf P_{X^+}\times\SO(V)$, where $\mathsf P_{X^+}$ is the parabolic subgroup of $\SO(W^+)$ stabilizing $X^+$.  We note
\[
G^+=\mathsf G^+(\BR),\quad H^+=\mathsf H^+(\BR),\quad P^+=\mathsf P^+(\BR).
\]

From the multiplicity-one  theorem (\cite{sun2012multiplicity}), we have 
\[
m(\pi_V\boxtimes \pi_W)\leqslant 1,
\] so it suffices to prove the following proposition.
\begin{pro}\label{pro: construction}
When $m(\pi_V\boxtimes \pi_W)\neq 0$ and $\sigma_{X^+}$ is a generic representation of $\GL(X^+)$, then one can construct a nonzero element in 
\[
\Hom_{H^+}((\sigma_{X^+}\ltimes \pi_W)\boxtimes\pi_V,1_{H^+}).
\]
\end{pro}

The main idea for proving this proposition is from the following theorem.  
\begin{thm} \label{thm: GSS2}(\cite[Proposition 4.9]{gourevitch2019analytic})
     For  a  Casselman-Wallach representation $\sigma^+$ of $P^+$, suppose
    \begin{enumerate}
        \item The complement $G^+-P^+H^+$ is the zero set of a polynomial $f^+$ on $G^+$ that is left $H^+$-invariant and right $(P^+,\psi_{P^+})$-equivariant for an algebraic character $\psi_{P^+}$ of $P^+$. 
        \item $H^+$ has finitely many orbits on the flag of a minimal parabolic subgroup of $G^+$
        \item $\sigma^+$ admits a nonzero $(P^+\cap H^+,\delta_{P^+\cap H^+}\delta_{H^+}^{-1})$-equivariant continuous linear functional, where $\delta_{P^+\cap H^+},\delta_{H^+}$ are the modular characters of $P^+\cap H^+$ and $H^+$ respectively.
    \end{enumerate}
    then $\Ind_{P^+}^{\CS,G^+}(\sigma^+)$ admits a nonzero $H^+$-invariant functional.
\end{thm}

We first verify (1)(2) in the set up of Proposition \ref{pro: construction}.
\begin{enumerate}
   \item 
   Fix a basis $v_1,\dots,v_n$ of $V$ and a basis $v_1^+,\dots,v_{r+1}^+$ of $X^+$.  For every $(g_{W^+},g_V)\in G^+$, $g\in G^+-P^+H^+$ if and only if $Xg_{W^+}\subset V$, equivalently, the $(n+1)\times (n+1+r)$-matrix \[A_g=\left[v_1g_V,\dots,v_{n}g_V,v_1^Xg_{W^+}^{-1},\dots,v_{r+1}^Xg_{W^+}^{-1}\right]\]  is of rank $n$. We let 
\begin{equation}
    f(g)=\det(A_gA_g^t),
\end{equation} then $f$ is  left $(P^+,\psi_{P^+})$-equivariant and right $H^+$-invariant, where 
\[
\psi_{P^+}(p_{X^{+}},g_V)=\det(g_{X^+})^2,\quad \text{ for } p_{X^+}=(g_{X^+},g_{W})\cdot n_{X^+}\in P_{W^+}
\text{ and }g_V\in \SO(V).
\]
  \item Since $G^+/H^+$ is an absolutely spherical variety (Section \ref{section: complex}), the Borel subgroup  has finitely many orbits, so the complexification of the minimal parabolic also has finitely many orbits. 
 Then condition (2) is a direct consequence of the finiteness of the first Galois cohomology for groups over local fields. 
 \end{enumerate}

Therefore, to complete the proof for Proposition \ref{pro: construction}, it suffices to construct a nonzero $(P^+\cap H^+,\delta_{P^+\cap H^+}\delta_{H^+}^{-1})$-equivariant continuous linear functional. 

As computed in Section \ref{section: multiplicity formula}, we have 
\[
H \bs P^+\cap H^+=N_{r+1}\bs R_{r,1},
\]
where $N_{r+1}$ and $R_{r,1}$ are the unipotent group and mirabolic group defined in Section \ref{section: multiplicity formula}. Hence, from \cite{soudry1993rankin}, the Rankin-Selberg integral
\[
F_s(v_{\pi_V},v_{\pi_W},v_{\sigma_{X^+}}):=\int_{P^+\cap H^+} \mu( \pi_V(p_{X^+})v_{\pi_V},v_{\pi_W}) \lambda(\sigma_{X+}({p_{X^+}})v_{\sigma_{X^+}})|\det(g_{X^+})|^{s} d(p_{X^+},p_{X^+})
\]
is absolutely convergent when $\Re(s)$ is large enough and extends to a meromorphic family in 
\[
F_s\in\Hom_{P^+\cap H^+}(\pi_V\boxtimes \pi_W\boxtimes\sigma_{X^+},|\det(g_X)|^{s-s_0}),
\]
where $s_0=\dim W-\dim X^+$, which is the real number satisfying $\delta_{P^+}(p_{X^+})=|\det(g_{X^+})|^{s_0}$. From \cite{gourevitch2019analytic}, we have 
\[
\frac{F_s}{(s-s_0)^{n_{s_0}}}\big |_{s=s_0}
\]
is a nonzero element
\[\Hom_{P^+\cap H^+}(\pi_V\boxtimes \pi_W\boxtimes\sigma_{X^+},1_{P^+\cap H^+}),
\]
where $n_{s_0}$ is the order of poles of $F_s$ at $s=s_0$. Therefore, we complete the proof for Proposition \ref{pro: construction}.
\bibliographystyle{alpha} 
\bibliography{cheng}
\end{document}